\documentclass[12pt]{amsart}

\usepackage{amssymb,amsbsy}
\usepackage{latexsym,color,longtable}
\usepackage{verbatim,euscript}
\usepackage{fullpage,array}
\usepackage{tikz}
\usetikzlibrary{arrows,shapes,trees}

\numberwithin{equation}{section}
\usepackage[colorlinks=true,linkcolor=blue,urlcolor=violet,citecolor=magenta]{hyperref}

\definecolor{my_color}{rgb}{0,0.5,0.5}
\definecolor{mixt}{rgb}{0.5,0.3,0.2}
\definecolor{darkgreen}{rgb}{0.09, 0.35, 0.27}
\definecolor{forest}{rgb}{0.13, 0.55, 0.13}

\input {cyracc.def}
\font\tencyr=wncyr10 
\font\tencyi=wncyi10 
\font\tencysc=wncysc10 
\def\rus{\tencyr\cyracc}
\def\rusi{\tencyi\cyracc}
\def\rusc{\tencysc\cyracc}

\catcode`\@=\active
\catcode`\@=11

\renewcommand{\@cite}[2]{[{{\bf #1}\if@tempswa , #2\fi}]}
\renewcommand{\@biblabel}[1]{[{\bf #1}]\hfill}
\catcode`\@=12

\newtheorem{thm}{Theorem}[section]
\newtheorem{lm}[thm]{Lemma}
\newtheorem{prop}[thm]{Proposition}

\theoremstyle{remark}
\newtheorem{rmk}[thm]{Remark}

\theoremstyle{definition}
\newtheorem{ex}[thm]{Example}
\newtheorem{df}{Definition}

\newcommand {\ah}{{\mathfrak a}}
\newcommand {\be}{{\mathfrak b}}
\newcommand {\ce}{{\mathfrak c}}

\newcommand {\De}{{\mathfrak D}}

\newcommand {\g}{{\mathfrak g}}
\newcommand {\h}{{\mathfrak h}}

\newcommand {\n}{{\mathfrak n}}

\newcommand {\p}{{\mathfrak p}}
\newcommand {\q}{{\mathfrak q}}

\newcommand {\te}{{\mathfrak t}}
\newcommand {\ut}{{\mathfrak u}}

\newcommand {\slno}{{\mathfrak {sl}}_{n+1}}
\newcommand {\sltn}{{\mathfrak {sl}}_{2n}}

\newcommand {\slv}{{\mathfrak {sl}}(\BV)}

\newcommand {\spv}{{\mathfrak {sp}}(\BV)}
\newcommand {\spn}{{\mathfrak {sp}}_{2n}}

\newcommand {\sov}{{\mathfrak {so}}(\BV)}

\newcommand {\soN}{{\mathfrak {so}}_{N}}

\newcommand {\eus}{\EuScript}

\newcommand {\gB}{{\eus B}}
\newcommand {\gC}{{\eus C}}
\newcommand {\gD}{{\eus D}}
\newcommand {\gH}{{\eus H}}
\newcommand {\gK}{{\eus K}}

\newcommand {\esi}{\varepsilon}
\newcommand {\ap}{\alpha}

\newcommand {\lb}{\lambda}

\newcommand {\tap}{{\tilde{\alpha}}}

\newcommand {\ca}{{\mathcal A}}

\newcommand {\N}{{\mathcal N}}
\newcommand {\co}{{\mathcal O}}
\newcommand {\cP}{{\mathcal P}}
\newcommand {\cQ}{{\mathcal Q}}

\newcommand {\BV}{{\mathbb{V}}}

\newcommand {\BC}{{\mathbb C}}
\newcommand {\BN}{{\mathbb N}}
\newcommand {\BQ}{{\mathbb Q}}
\newcommand {\BZ}{{\mathbb Z}}

\newcommand {\ad}{{\mathrm{ad\,}}}
\newcommand {\ads}{ {\mathrm{ad}}^* }

\newcommand {\hot}{{\mathsf{ht}}}
\newcommand {\htt}{\widetilde{\textsl{ht\,}}}
\newcommand {\ind}{{\mathrm{ind\,}}}
\newcommand {\Lie}{{\mathrm{Lie\,}}}
\newcommand {\Ker}{{\mathsf{Ker\,}}}
\newcommand {\Ima}{{\mathsf{Im\,}}}

\newcommand {\rk}{{\mathsf{rk\,}}}

\newcommand {\spec}{{\mathsf{spec}}}
\newcommand {\supp}{{\mathsf{supp}}}

\newcommand {\tr}{{\mathrm{tr}}}
\newcommand {\trdeg}{{\mathrm{trdeg\,}}}

\newcommand {\tri}{{\mathfrak{sl}}_2}

\newcommand {\GR}[2]{{\textrm{{\sf\bfseries #1}}}_{#2}}
\newcommand {\GRt}[2]{ {\widetilde{\textrm{\sf\bfseries #1} } }_{#2}   }

\newcommand {\ov}{\overline}
\newcommand {\un}{\underline}

\newcommand {\beq}{\begin{equation}}
\newcommand {\eeq}{\end{equation}}

\newcommand {\blb}{\boldsymbol{\lambda}}
\newcommand {\bi}{{\boldsymbol{i}}}

\newcommand {\spx}{\spec(x_\gK)}

\newcommand{\curle}{\preccurlyeq}
\renewcommand{\le}{\leqslant}
\renewcommand{\ge}{\geqslant}
\renewcommand{\lg}{\langle}
\newcommand{\rg}{\rangle}

\newfam\Bbbfam%
\font\Bbbfont=msbm10 scaled 1200%
\font\Bbbsmallfont=msbm8%
\textfont\Bbbfam=\Bbbfont\scriptfont\Bbbfam=\Bbbsmallfont

\begin{document}
\setlength{\parskip}{2pt plus 4pt minus 0pt}
\hfill {\scriptsize May 20, 2022} 
\vskip1ex

\title[Properties of the cascade]{Combinatorial and geometric constructions associated with the 
Kostant cascade}
\author{Dmitri I. Panyushev}
\address{
Institute for Information Transmission Problems of the R.A.S., 
Moscow 127051, Russia}
\email{panyushev@iitp.ru}
\thanks{This research was funded by RFBR, project {\rus N0} 20-01-00515.}
\dedicatory{To Alexander Grigorievich Elashvili on the occasion of his 80th birthday}
\begin{abstract}
Let $\g$ be a complex simple Lie algebra and $\be=\te\oplus\ut^+$ a fixed Borel subalgebra.
Let $\Delta^+$ be the set of positive roots associated with $\ut^+$ and $\gK\subset\Delta^+$ the 
Kostant cascade. We elaborate on some constructions related to $\gK$ and applications of $\gK$. This includes the cascade element $x_\gK$ in the Cartan subalgebra $\te$ and properties of certain objects naturally associated with $\gK$: an abelian ideal of $\be$, a nilpotent $G$-orbit in $\g$, and an involution of $\g$. 
\end{abstract}
\maketitle

\tableofcontents

\section{Introduction}
\noindent
Let $G$ be a simple algebraic group with $\Lie(G)=\g$. Fix a triangular decomposition 
$\g=\ut^+\oplus\te\oplus\ut^-$. Then $\Delta$ is the root system of $(\g,\te)$ and $\Delta^+$ is the set 
of positive roots corresponding to $\ut^+$. The {\it Kostant cascade\/} is a set $\gK$ of strongly 
orthogonal roots in $\Delta^+$ that is constructed recursively starting with the highest root 
$\theta\in\Delta^+$, see Section~\ref{sect:prelim-kaskad}. The construction of cascade goes back to 
B.\,Kostant, who used it for studying the center of the enveloping algebra of $\ut^+$. 
His construction is prominently used in some articles afterwards~\cite{GS,jos77, lw}, but Kostant's own 
publications related to the cascade appear some 40 years later~\cite{ko12,ko13}. The cascade is also 
crucial for computing the index of seaweed subalgebras of simple Lie algebras~\cite{ty04,jos}.
\\ \indent
Ever since I learned from A.\,Elashvili about the cascade at the end of 80s, I was fascinated by this 
structure. Over the years, I gathered a number of results related to the occurrences of $\gK$ in various 
problems of Combinatorics, Invariant Theory, and Representation Theory. In~\cite{p22}, I give an 
application of $\gK$ to the problem of classifying the nilradicals of parabolic subalgebras of $\g$ that 
admit a commutative polarisation. (General results on commutative polarisations are due to Elashvili and 
Ooms, see~\cite{ag03}.) Some other observations appear in this article. 

Let $\Pi\subset\Delta^+$ be the set of simple roots. If $\gamma=\sum_{\ap\in \Pi}a_\ap\ap$, then
$[\gamma:\ap]=a_\ap$ and $\htt(\gamma)=\sum_{\ap\in \Pi}[\gamma:\ap]$ is the {\it height\/} of 
$\gamma$. The set of positive roots $\Delta^+$ is a poset with respect to the root order ``$\curle$'', and 
$\gK=\{\beta_1,\dots,\beta_m\}$ inherits this structure so that $\theta=\beta_1$ is the unique maximal 
element of $\gK$. 

In Section~\ref{sect:comb-prop}, we define a rational element of $\te$ associated with $\gK$.  Let $(\ ,\,)$ 
denote the restriction of the Killing form on $\g$ to $\te$. As usual, $\te$ and $\te^*$ are identified via 
$(\ ,\,)$ and $\te^*_\BQ$ is the $\BQ$-linear span of $\Delta$.
The {\it cascade element\/} of $\te_\BQ\simeq\te^*_\BQ$ is
\beq       \label{eq:casc-elem}
       x_\gK=\sum_{i=1}^m\frac{\beta_i}{(\beta_i,\beta_i)}=\frac{1}{2}\sum_{i=1}^m{\beta_i}^\vee .
\eeq
The numbers $\gamma(x_\gK)$, $\gamma\in\Delta^+$, are the eigenvalues of $\ad x_\gK$ on
$\ut^+$, and we say that they form the spectrum of $x_\gK$ on $\Delta^+$.  
It follows from~\eqref{eq:casc-elem} that $\gamma(x_\gK)\in \frac{1}{2}\BZ$ and $\beta(x_\gK)=1$ for 
any $\beta\in\gK$. We prove that $-1\le \gamma(x_\gK)\le 2$ for any $\gamma\in\Delta^+$ and if $\g$ is 
not of type $\GR{A}{2p}$, then the eigenvalues are integral (Theorem~\ref{thm:spektr-fonin}). It is also 
shown that the spectrum of $x_\gK$ on $\Delta^+\setminus\gK$ is symmetric relative to $1/2$, which 
means that if $m_\lb$ is the multiplicity of the eigenvalue $\lb$, then $m_\lb=m_{1-\lb}$. If $\theta$ is a 
fundamental weight and $\ap\in\Pi$ is the unique root such that $(\theta,\ap)\ne 0$, then $\ap$ is long 
and we prove that $\ap(x_\gK)=-1$. On the other hand, if $\theta$ is not fundamental, then $x_\gK$  
appears to be dominant. Let $\cQ$ (resp. $\cQ^\vee$) denote the {\it root} (resp. {\it coroot}) {\it lattice} in 
$\te^*_\BQ$. The corresponding dual lattices are 
\\[.6ex]
\centerline{
the coweight lattice $\cP^\vee:=\cQ^*$ \ \& \ the weight lattice $\cP:=(\cQ^\vee)^*$.}
\\[.6ex]
Hence $x_\gK\in \cP^\vee$ unless $\g$ is of type $\GR{A}{2p}$. Here $\cQ^\vee\subset\cP^\vee$ and
we prove that $x_\gK\in\cQ^\vee$ if and only if every self-dual representation of $\g$ is orthogonal (Section~\ref{sect:self-dual}).

In~\cite[Section\,3]{ooms}, A.\,Ooms describes an interesting feature of the Frobenius Lie algebras.
Let $\q=\Lie(Q)$ be a Frobenius algebra and $\xi\in\q^*$ a regular linear form, i.e., $\q^\xi=\{0\}$. Then 
the  Kirillov form $\gB_\xi$ is non-degenerate and it yields a linear isomorphism $\bi_\xi:\q^*\to \q$. 
Letting $x_\xi=\bi_\xi(\xi)$, Ooms proves that $(\ad x_\xi)^*=1- \ad x_\xi$, where $(\ad x_\xi)^*$ is the 
adjoint operator w.r.t.{} $\gB_\xi$. This implies that the spectrum of $\ad x_\xi$ on $\q$ is symmetric 
relative to $1/2$. We say that $x_\xi\in\q$ is the {\it Ooms element\/} associated with $\xi\in\q^*_{\sf reg}$. 
Let $\te_\gK\subset\te$ be the $\BC$-linear span of $\gK$. Then $\be_\gK=\te_\gK\oplus\ut^+$ is a
Frobenius Lie algebra, i.e., $\ind\be_\gK=0$, see~\cite[Sect.\,5]{p22}. 
In Section~\ref{sect:frob}, we prove that 
\begin{itemize}
\item 
if $\q$ is an algebraic Lie algebra, then any Ooms element $x_\xi\in\q$ is 
semisimple;
\item 
$x_\gK$ is an Ooms element for the Frobenius Lie algebra 
$\be_\gK=\te_\gK\oplus\ut^+$.
\end{itemize}
The latter provides a geometric explanation for the symmetry of the spectrum of $x_\gK$ on 
$\Delta^+\setminus\gK$.


By~\cite{ko98}, one associates an abelian ideal of $\be$, $\ah_z$, to any $z\in\te$ such that 
$\gamma(z)\in\{-1,0,1,2\}$ for all $\gamma\in\Delta^+$. Namely, let 
$\Delta^+_z(i)=\{\gamma\in\Delta^+ \mid \gamma(z)=i\}$. There is a unique $w_z\in W$ such that the inversion 
set of $w_z$, $\eus N(w_z)$, equals $\Delta^+_z(1)\cup\Delta^+_z(2)$, and then $\Delta_{\lg z\rg}:=
w_z\bigl(\Delta^+_z(-1)\cup -\Delta^+_z(2)\bigr)$ is the set of roots of $\ah_z$. We notice that $w_z(z)$ is 
anti-dominant and $w_z$ is the element of minimal length having such property. Kostant's construction 
applies to $z=x_\gK$ unless $\g$ is of type $\GR{A}{2p}$ and we obtain a complete description of 
$w_\gK:=w_{x_\gK}$ and $\ah_\gK:=\ah_{x_\gK}$ (Sections~\ref{sect:ab-ideal},\,\ref{sect:w_K&perebor}). 
In this setting, we prove that $w_\gK(\theta)\in -\Pi$ and $-w_\gK(x_\gK)$ is a fundamental coweight. 
That is, if $w_\gK(\theta)=-\ap_j$, then $-w_\gK(x_\gK)=2\varpi_j/(\ap_j,\ap_j)=:\varpi_j^\vee$.
Since $[\be,\ah_\gK]\subset\ah_\gK$, the set of roots $\Delta_{\lg \gK\rg}=\Delta(\ah_\gK)$ is an upper 
ideal of the poset $(\Delta^+,\curle)$. Therefore, $\Delta_{\lg \gK\rg}$ is fully determined by the set of 
{\sl minimal} elements of $\Delta_{\lg \gK\rg}$ or the set of {\sl maximal} elements of 
$\Delta^+\setminus\Delta_{\lg \gK\rg}$. Letting $\Delta^+_\gK(i)=\Delta^+_{x_\gK}(i)$ and
$\Pi_\gK(i)=\Pi\cap\Delta^+_\gK(i)$, we prove that
\begin{enumerate}
\item $\min(\Delta_{\lg \gK\rg})=w_\gK(\Pi_\gK(-1))$ \ and \ $\max(\Delta^+\setminus \Delta_{\lg \gK\rg})=-w_\gK(\Pi_\gK(1))$;
\item if $d_\gK=1+\sum_{\ap\in \Pi_\gK({\ge}0)}[\theta:\ap]$, then  $\Delta_{\lg \gK\rg}=\{\gamma\mid \htt(\gamma)\ge d_\gK\}$.
\end{enumerate}
In order to verify (2), we use explicit formulae for $w_\gK$. To get such formulae, we exploit a description
of $w_\gK^{-1}(\Pi)$ (Theorem~\ref{thm:w^-1}). Then we check directly that $\htt(w_\gK(\ap))=d_\gK$ for 
any $\ap\in\Pi_\gK(-1)$ and that
$\#\Pi_\gK(-1)=\#\{\gamma\mid \htt(\gamma)= d_\gK\}$.

In Section~\ref{sect7:involution}, we naturally associate an involution $\sigma_\gK\in \mathsf{Aut}(\g)$ to 
$\gK$ 
if $x_\gK\in\cP^\vee$, i.e., $\g$ is not of type $\GR{A}{2p}$. It is proved $\sigma_\gK$ is the unique, up to 
conjugation, inner involution such that the $(-1)$-eigenspace of $\sigma_\gK$ contains a regular nilpotent
element of $\g$. 

For any simple Lie algebra ($\mathfrak{sl}_{2p+1}$ included), we construct the nilpotent $G$-orbit 
associated with $\gK$, see Section~\ref{sect:nilp-orb}. 
Let $e_\gamma\in \g_\gamma$ be a nonzero root vector 
($\gamma\in\Delta$). Then $e_\gK=\sum_{\beta\in\gK}e_\beta\in \g$ is nilpotent and the orbit 
$\co_\gK=G{\cdot}e_\gK$ does not depend on the choice of root vectors. Properties of $\co_\gK$ 
essentially depend on whether $\theta$ is fundamental or not. We prove that if $\theta$ is fundamental 
then $(\ad e_\gK)^5=0$ and $(\ad e_\gK)^4\ne 0$; whereas if $\theta$ is not fundamental then 
$(\ad e_\gK)^3=0$ (and, of course, $(\ad e_\gK)^2\ne 0$). By~\cite{p94}, this means
that $\co_\gK$ is spherical if and only if $\theta$ is not fundamental. Here $[x_\gK,e_\gK]=e_\gK$ and
$x_\gK\in \Ima(\ad e_\gK)$. Therefore, $2x_\gK$ is a {\it characteristic\/} of $e_\gK$ and the weighted Dynkin diagram of $\co_\gK$, $\gD(\co_\gK)$ is determined by the dominant element in 
$W{\cdot}(2x_\gK)$. Actually, this dominant element is $-2w_\gK(x_\gK)$, and if $\g\ne\mathfrak{sl}_{2p+1}$, then $-2w_\gK(x_\gK)=2\varpi_j^\vee$, cf. above. Hence, in these cases, $\gD(\co_\gK)$ has the unique nonzero label on the node $\ap_j$ and $\co_\gK$ is even,  
cf. Tables~\ref{table:O-class}, \ref{table:O-exc}.

\un{Main notation}. 
Throughout, $G$ is a simple algebraic group with $\g=\Lie(G)$. Then
\begin{itemize}
\item[--] $\be$ is a fixed Borel subalgebra of $\g$ and $\ut^+=\ut=[\be,\be]$;
\item[--] $\te$ is a fixed Cartan subalgebra in $\be$ and $\Delta$ is the root system of $(\g,\te)$; 
\item[--] $\Delta^\pm$ is the set of roots corresponding to $\ut^\pm$;  
\item[--] $\Pi=\{\ap_1,\dots,\ap_n\}$ is the set of simple roots in $\Delta^+$ and the corresponding fundamental weights are $\varpi_1,\dots,\varpi_n$;
\item[--] $\te^*_\BQ$ is the $\BQ$-vector subspace of $\te^*$ spanned by $\Delta$, and $(\ ,\, )$ is the 
positive-definite form on $\te^*_\BQ$ induced by the Killing form on $\g$; as usual, 
$\gamma^\vee=2\gamma/(\gamma,\gamma)$ for $\gamma\in\Delta$.
\item[--] For each $\gamma\in\Delta$, $\g_\gamma$ is the root space in $\g$ and 
$e_\gamma\in\g_\gamma$ is a nonzero vector;
\item[--]  If $\ce\subset\ut^+$ is a $\te$-stable subspace, then $\Delta(\ce)\subset \Delta^+$ is the 
set of roots of $\ce$;
\item[--]  $\theta$ is the highest root in $\Delta^+$;
\item[--]  $W\subset GL(\te)$ is the Weyl group.
\end{itemize}
Our main references for (semisimple) algebraic groups and Lie algebras are~\cite{VO,t41}. In explicit 
examples related to simple Lie algebras, the Vinberg--Onishchik numbering of simple roots and 
fundamental weights is used, see e.g.~\cite[Table\,1]{VO} or \cite[Table\,1]{t41}.

\section{Preliminaries on root systems and the Kostant cascade} 
\label{sect:prelim-kaskad}
\noindent
We identify $\Pi$ with the vertices of the Dynkin diagram of $\g$. For any $\gamma\in\Delta^+$, let 
$[\gamma:\ap]$ be the coefficient of $\ap\in\Pi$ in the expression of $\gamma$ via $\Pi$. The 
{\it support\/} of $\gamma$ is $\supp(\gamma)=\{\ap\in\Pi\mid [\gamma:\ap]\ne 0\}$ and the {\it height 
of\/}  $\gamma$ is $\htt(\gamma)=\sum_{\ap\in\Pi}[\gamma:\ap]$. As is well known, 
$\supp(\gamma)$ is a connected subset of the Dynkin diagram. For instance, $\supp(\theta)=\Pi$
and $\supp(\ap)=\{\ap\}$. A root $\gamma$ is {\it long}, if $(\gamma,\gamma)=(\theta,\theta)$. 
We write $\Delta_l$ (resp. $\Delta_s$) for the set of long (resp. short) roots in $\Delta$. In the simply-laced case, $\Delta_s=\varnothing$.
\\ \indent
Let ``$\curle$'' denote the {\it root order\/} in $\Delta^+$, i.e., we write 
$\gamma\curle\gamma'$ if $[\gamma:\ap]\le [\gamma':\ap]$ for all $\ap\in\Pi$. Then $\gamma'$ covers
$\gamma$ if and only if $\gamma'-\gamma\in\Pi$, which implies that $(\Delta^+,\curle)$ is a graded 
poset. Write $\gamma\prec\gamma'$ if $\gamma\curle\gamma'$ and $\gamma\ne\gamma'$. An 
{\it upper ideal\/} of $(\Delta^+,\curle)$ is a subset $I$ such that if $\gamma\in I$ and 
$\gamma\prec\gamma'$, then $\gamma'\in I$. Therefore, $I$ is an upper ideal if and only if 
$\ce=\bigoplus_{\gamma\in I} \g_\gamma$ is a $\be$-ideal of $\ut$ (i.e., $[\be,\ce]\subset\ce$). 

For a dominant weight $\lb\in\te^*_\BQ$, set 
$\Delta^\pm_\lb=\{\gamma\in\Delta^\pm\mid (\lb,\gamma)=0 \}$ and
$\Delta_\lb=\Delta^+_\lb\cup \Delta^-_\lb$. Then $\Delta_\lb$ is the root system of a semisimple 
subalgebra $\g_\lb\subset \g$ and $\Pi_\lb=\Pi\cap\Delta^+_\lb$ is the set of simple roots
in $\Delta^+_\lb$. Set $\Delta_\lb^{{>}0}=\{\gamma\in \Delta^+\mid (\lb,\gamma)>0\}$. Then
$\Delta^+ =\Delta^+_\lb \sqcup \Delta_\lb^{{>}0}$ and
\begin{itemize}
\item \ $\p_\lb=\g_\lb+\be$ is a standard parabolic subalgebra of $\g$; 
\item \ the set of roots for the nilradical $\n_\lb=\p_\lb^{\sf nil}$ is $\Delta_\lb^{{>}0}$; it is also denoted by
$\Delta(\n_\lb)$.
\end{itemize}
If $\lb=\theta$, then  the nilradical $\n_\theta$ is a {\it Heisenberg Lie algebra}. 
In this case,  
$\eus H_\theta:=\Delta(\n_\theta)$ is said to be the {\it Heisenberg subset\/} (of $\Delta^+$).

The construction of the Kostant cascade $\gK$ in $\Delta^+$ is recalled below, see also~\cite[Sect.\,2]{jos77}, \cite[Section\,3a]{lw}, and \cite{ko12,ko13}. Whenever we wish to stress that
$\gK$ is associated with $\g$, we write $\gK(\g)$ for it.
\\ \indent
{\it\bfseries 1.} \  We begin with $(\g\lg 1\rg,\Delta\lg 1\rg,\beta_1)=(\g,\Delta,\theta)$ and consider the 
(possibly reducible) root system $\Delta_\theta$. The highest root $\theta=\beta_1$ is the unique element 
of the {\bf first} (highest) level in $\gK$.
Let $\Delta_\theta=\bigsqcup_{j=2}^{d_2} \Delta\lg j\rg$ be the decomposition into irreducible root 
systems and $\Pi\lg j\rg =\Pi\cap \Delta\lg j\rg$. Then
$\Pi_\theta=\bigsqcup_{j=2}^{d_2}\Pi\lg j\rg$ and $\{\Pi\lg j\rg\}$ are the connected components of
$\Pi_\theta\subset \Pi$.
\\ \indent
{\it\bfseries 2.} \ Let $\g\lg j\rg$ be the simple subalgebra of $\g$ with root system $\Delta\lg j\rg $. Then
$\g_\theta=\bigoplus_{j=2}^{d_2}\g\lg j\rg $.  Let $\beta_{j}$ be the highest root in $\Delta\lg j\rg ^{+}=\Delta\lg j\rg \cap\Delta^+$. The roots $\beta_2,\dots,\beta_{d_2}$ are 
the {\it descendants\/} of $\beta_1$, and they form the {\bf second} level of $\gK$.
Note that $\supp(\beta_j)=\Pi\lg j\rg $, hence different descendants have disjoint supports. 
\\ \indent
{\it\bfseries 3.} \ Making the same step with each pair $(\Delta\lg j\rg,\beta_j)$, $j=2,\dots,d_2$, we get a 
collection of smaller simple subalgebras inside each $\g\lg j\rg$ and smaller irreducible root systems 
inside $\Delta\lg j\rg$. This provides the descendants for each $\beta_j$ ($j=2,\dots,d_2$), i.e., the 
elements of the {\bf third} level in $\gK$. And so on...
\\ \indent
{\it\bfseries 4.} \ This procedure eventually terminates and yields a maximal set 
$\gK=\{\beta_1,\beta_2,\dots,\beta_m\}$ of {\it strongly orthogonal\/} roots in $\Delta^+$. (The latter 
means that $\beta_i\pm\beta_j\not\in\Delta$ for all $i,j$). We say that $\gK$ is the {\it Kostant cascade} 
in $\Delta^+$.  

\noindent
Thus, each $\beta_i\in\gK$ occurs as the highest root of a certain irreducible root system $\Delta\lg i\rg$ 
inside $\Delta$ such that $\Pi\lg i\rg=\Pi\cap\Delta\lg i\rg^+$ is a basis for $\Delta\lg i\rg$. 

We think of $\gK$ as poset such that $\beta_1=\theta$ is the unique maximal element and each 
$\beta_i$ covers exactly its own descendants. If $\beta_j$ is a descendant of $\beta_i$, then 
$\beta_j\prec \beta_i$ in $(\Delta^+,\curle)$
and $\supp(\beta_j)\varsubsetneq\supp(\beta_i)$, while different descendants of 
$\beta_i$ are not comparable in $\Delta^+$. Therefore the poset structure of $\gK$ is the restriction of the  root order in $\Delta^+$.
The resulting poset $(\gK, \curle)$ is called the {\it cascade poset}. The numbering of $\gK$ is not 
canonical. We only require that it is a linear extension of $(\gK,\curle)$, i.e., if $\beta_j$ is a 
descendant of $\beta_i$, then $j>i$. 

Using the decomposition $\Delta^+ =\Delta^+_\theta \sqcup \eus H_\theta$ and induction on $\rk\g$, one
readily obtains the disjoint union determined by $\gK$: 
\beq            \label{eq:decomp-Delta}
     \Delta^+=\bigsqcup_{i=1}^m \eus H_{\beta_i}=\bigsqcup_{\beta\in\gK}\eus H_\beta ,
\eeq
where $\eus H_{\beta_i}$ is the Heisenberg subset in $\Delta\lg i\rg^+$ and 
$\eus H_{\beta_1}=\eus H_\theta$. The geometric counterpart of this decomposition is the direct sum of 
vector spaces
\[
      \ut^+=\bigoplus_{i=1}^m  \h_i ,
\]
where $\h_i$ is the Heisenberg Lie algebra in $\g\lg i\rg$, with $\Delta(\h_i)=\eus H_{\beta_i}$. In 
particular, $\h_1=\n_\theta$. 

For any $\beta\in\gK$, set $\Phi(\beta)=\Pi\cap \eus H_{\beta}$. Then 
$\Pi=\bigsqcup_{\beta\in\gK} \Phi(\beta)$ and $\Phi$ is thought of as a map from $\gK$ to $2^{\Pi}$. 
Our definition of subsets $\Phi(\beta_i)$ yields the well-defined map  $\Phi^{-1}: \Pi\to \gK$, 
where $\Phi^{-1}(\ap)=\beta_i$ if $\ap\in \Phi(\beta_i)$. Note that $\ap\in\Phi(\Phi^{-1}(\ap))$. We also 
have $\#\Phi(\beta_i)\le 2$ and $\#\Phi(\beta_i)= 2$ if and only if the root system $\Delta\lg i\rg$ is of 
type $\GR{A}{n}$ with $n\ge 2$.

The cascade poset $(\gK,\curle)$ with the set $\Phi(\beta)$ attached to each $\beta$ is
called the {\it marked cascade poset} ({\sf MCP}).  In~\cite{p22}, we use $(\gK, \curle, \Phi)$ for
describing the nilradicals of parabolic subalgebras that admit a commutative polarisation.
The Hasse diagrams of $(\gK, \curle)$ are presented in Appendix~\ref{sect:tables}, where the Cartan 
label of the simple Lie algebra $\g\lg j\rg$ is attached to the node $\beta_j$. 
These diagrams (without Cartan labels) appear already in~\cite[Section\,2]{jos77}. 
\\  \indent
Let us gather some properties of $(\gK,\curle)$ that either are explained above or easily follow from the construction. 

\begin{lm}    \label{lm:K-svojstva}   
Let $(\gK, \curle, \Phi)$ be the {\sf MCP} for a simple Lie algebra $\g$.
\begin{enumerate}
\item The partial order in $\gK$ coincides with the restriction to $\gK$ of the root order in $\Delta^+$;
\item  $\beta_i,\beta_j\in\gK$ are comparable if and only if\/ 
$\supp(\beta_i)\cap\supp(\beta_j)\ne \varnothing$; and then one support is properly contained in the 
other;
\item each $\beta_j$, $j\ge 2$, is covered by a unique element of $\gK$;
\item for any $\beta_j\in\gK$, the interval 
$[\beta_j,\beta_1]_{\gK}=\{\nu\in\gK\mid \beta_j\curle\nu\curle\beta_1\} \subset\gK$ is a chain.
\item  For $\ap\in\Pi$, we have $\ap\in\Phi(\beta_i)$ if and only if $(\ap, \beta_i)>0$.
\end{enumerate}
\end{lm}
Clearly, $\#\gK\le \rk\g$ and  the equality holds if and only if each $\beta_i$ is a multiple of a fundamental 
weight for $\g\lg i\rg$. Recall that $\theta$ is a multiple of a fundamental weight of $\g$ if and only if $\g$ 
is not of type $\GR{A}{n}$, $n\ge 2$. It is well known that the following conditions are equivalent:
{\sl (1)} $\ind\be=0$; {\sl (2)} $\#\gK=\rk\g$, see e.g.~\cite[Prop.\,4.2]{ap97}. 
This happens exactly if $\g$ is not of type $\GR{A}{n} \ (n\ge 2)$, 
$\GR{D}{2n+1} \ (n\ge 2)$, $\GR{E}{6}$. Then $\Phi$ yields a bijection between $\gK$ and $\Pi$.

For future reference, we record the following observation.
\begin{lm}         \label{lm:1-short}
If\/ $\g$ is of type $\GR{B}{2k+1}$ or $\GR{G}{2}$, then $\gK$ contains a unique short root, which is 
simple. In all other cases,  all elements of\/ $\gK$ are long. 
\end{lm}

Write $r_\gamma\in W$ for the reflection relative to $\gamma\in\Delta$.
\begin{prop}[{\cite[Prop.\,1.10]{ko12}}]   \label{prop:longest&K}
The product $\omega_0:=r_{\beta_1}{\cdot}\ldots{\cdot}r_{\beta_m}$
does not depend on the order of factors and it is the longest element of\/ $W$ (i.e., 
$\omega_0(\Delta^+)=\Delta^-$). In particular, $\omega_0(\beta_i)=-\beta_i$ for each $i$. 
\end{prop}
It follows from this that $\omega_0=-1$ if and only if $m=\rk\g$.

\section{The cascade element of a Cartan subalgebra}
\label{sect:comb-prop}

\noindent
In this section, we define a certain element of $\te$ associated with the cascade $\gK$ and consider its 
properties related to $\Delta$. As usual, we identify $\te$ and $\te^*$ using the restriction of the Killing form to $\te$. 

\begin{df}   \label{def:casc-elem}
The {\it cascade element\/} of $\te$ is the unique element 
$x_\gK\in \lg\beta_1,\dots,\beta_m\rg_\BQ\subset \te_\BQ$ such that $\beta_i(x_\gK)=1$ for each $i$. 
\end{df}
Since the roots $\{\beta_i\}$ are pairwise orthogonal, we have
\beq       \label{eq:x_K}
       x_\gK=\sum_{i=1}^m\frac{\beta_i}{(\beta_i,\beta_i)}=\frac{1}{2}\sum_{i=1}^m \beta_i^\vee .
\eeq
Therefore, $\gamma(x_\gK)\in \frac{1}{2}\BZ$ for any $\gamma\in \Delta$, and it follows from 
Prop.~\ref{prop:longest&K} that $\omega_0(x_\gK)=-x_\gK$. If $\gK\subset\Delta_l$, 
then one can also write $\displaystyle x_\gK=\frac{1}{(\theta,\theta)}\sum_{i=1}^m \beta_i$.  

\textbullet \ \ It is a typical pattern related to $\gK$ and $x_\gK$  that a certain property holds for series $\GR{A}{n}$ 
and $\GR{C}{n}$, but does not hold for the other simple types. The underlying reason is that 
\\[.6ex]
\centerline{\it 
$\theta$ is a fundamental weight if and only if\/ $\g$ is {\bf not} of type $\GR{A}{n}$ or $\GR{C}{n}$.
}
\\[.6ex]
(Recall that $\theta=\varpi_1+\varpi_n$ for $\GR{A}{n}$ and $\theta=2\varpi_1$ for $\GR{C}{n}$.)
It is often possible to prove that a property does not hold if $\theta$ is fundamental, and then 
directly verify that that property does hold for $\slno$ and $\spn$ (or vice versa). 
\\ \indent
\textbullet \ \ Yet another pattern is
that one has to often exclude the series $\GR{A}{2n}$ from consideration. The reason is that 
\\[.6ex]
\centerline{\it 
the \ \emph{Coxeter number} of\/ $\g$, $\mathsf{h}=\mathsf{h}(\g)$, is odd if and only if\/ $\g$ is of type 
$\GR{A}{2n}$. 
}
\\[.6ex] (The same phenomenon occurs also in the context of the McKay correspondence.)
Recall that $\mathsf{h}=1+\sum_{\ap\in\Pi}[\theta:\ap]=1+\htt(\theta)$.

To get interesting properties of $x_\gK$, we need some preparations.
Set $n_\ap=[\theta:\ap]$, i.e., $\theta=\sum_{\ap\in\Pi}n_\ap\ap$. Suppose that $\theta$ is a fundamental weight, and let $\tap$
be the unique simple root such that $(\theta,\tilde\ap)\ne 0$. Then $(\theta,{\tap}^\vee)=1$ and 
$(\theta^\vee,{\tap})=1$, hence $\tilde\ap$ is long. Next, 
\[
    (\theta,\theta)=(\theta, \sum_{\ap\in\Pi}n_\ap\ap) =(\theta, n_\tap\tap)=\frac{1}{2}n_\tap(\tap,\tap),
\]
which means that $n_\tap=2$. Let $\tilde\Pi$ be the set of simple roots that are adjacent to $\tap$ in the
Dynkin diagram. Since $\tap$ is long, one also has
$\tilde\Pi=\{\nu\in\Pi\mid (\nu,\tap^\vee)=-1\}$. Then
\[
  1= (\theta,\tap^\vee)=n_\tap(\tap,\tap^\vee)+\sum_{\nu\in \tilde\Pi}n_\nu(\nu,\tap^\vee)=4-\sum_{\nu\in \tilde\Pi}n_\nu .
\]
Hence $\sum_{\nu\in\tilde\Pi}n_\nu=3$ and $\#(\tilde\Pi)\le 3$. Set $J=\{i\in [1,m]\mid (\tap,\beta_i)<0\}$.
Then $1\not\in J$ and we proved in~\cite[Sect.\,6]{p05} that 
\beq    \label{eq:tap}
     \tap=\frac{1}{2}\left(\theta-\sum_{i\in J}\frac{(\tap,\tap)}{(\beta_i,\beta_i)}\beta_i\right)=:
     \frac{1}{2}(\theta-\sum_{i\in J}c_i\beta_i) \ \text{ and } \ \sum_{i\in J}c_i=3 ,
\eeq
see~\cite[Lemma\,6.5]{p05}. Here $c_i\in \BN$ and therefore $\#J\le 3$. Set $\tilde{\gK}:=\{\beta_i \mid i\in J\}$.

We say that $x\in\te$ is {\it dominant}, if $\gamma(x)\ge 0$ for all $\gamma\in\Delta^+$.

\begin{lm}    \label{lm:not-dominant}
If $\theta$ is fundamental, then $\tap(x_\gK)=-1$ and $(\theta-\tap)(x_\gK)=2$.
In particular, $x_\gK\in\te$ is not dominant.
\end{lm}
\begin{proof}
Take $\gamma=\tap$. Using \eqref{eq:x_K}  and \eqref{eq:tap}, we obtain
\[
  \tap(x_\gK)= \frac{1}{2}\left[\frac{(\theta,\theta)}{(\theta,\theta)}
  -\sum_{i\in J}\frac{c_i(\beta_i,\beta_i)}{(\beta_i,\beta_i)}\right]=(1-3)/2=-1 .  
\] 
Then $(\theta-\tap)(x_\gK)=1+1=2$.
\end{proof}
Conversely, if $\theta$ is {\bf not} fundamental, then the example below shows that $x_\gK$ {\bf is} 
dominant.
\begin{ex}             \label{ex:sl-sp}
(1) For $\g=\slno$, one has $\beta_i=\esi_i-\esi_{n+2-i}$ with $i=1,\dots,m=[(n+1)/2]$. Here $\sum_{j=1}^{n+1}\esi_j=0$ and $\varpi_j=\esi_1+\dots+\esi_j$. Hence
\[
   (\theta,\theta){\cdot}x_\gK=\sum_{i=1}^m \beta_i=\begin{cases} 2\varpi_p, & \text{if } \ n=2p{-}1 \\
   \varpi_p+\varpi_{p+1}, & \text{if } \  n=2p . \end{cases}
\]   
In the matrix form, one has $x_\gK=\begin{cases} 
\mathsf{diag}(1/2,\dots,1/2,-1/2,\dots,-1/2), &  \text{if } \ n=2p{-}1 \\
\mathsf{diag}(1/2,\dots, 1/2,0,-1/2,\dots,-1/2), &  \text{if } \ n=2p \ . 
\end{cases}$

(2) For $\spn$, one has $\beta_i=2\esi_i$ with $i=1,\dots,m=n$. Hence 
$(\theta,\theta){\cdot}x_\gK=\sum_{i=1}^n 2\esi_i=2\varpi_n$.
\end{ex}

Consider the multiset $\eus{M}_\gK$ of values 
$\{\gamma(x_\gK)\mid \gamma\in \Delta^+\setminus \gK\}$. That is, each value $d$ is taken with 
multiplicity $m_d=\# \eus R_d$, where 
$\eus R_d=\{\gamma\in\Delta^+\setminus \gK \mid  \gamma(x_\gK)=d\}$. 

\begin{lm}     \label{lm:symmetry-value}
For any $d$, there is a natural bijection between the sets $\eus R_d$ and $\eus R_{1-d}$. In particular,
$m_d=m_{1-d}$, i.e., the multiset $\eus{M}_\gK$ is symmetric w.r.t. $1/2$. 
\end{lm}
\begin{proof}
For any $\gamma\in \Delta^+\setminus \gK$, there is a unique $j\in \{1,\dots,m\}$ such that 
$\gamma\in\Delta(\h_j)\setminus\{\beta_j\}$, see~\eqref{eq:decomp-Delta}. Then $\beta_j-\gamma\in \Delta(\h_j)$ and
$\gamma(x_\gK)+(\beta_j-\gamma)(x_\gK)=1$.
\end{proof}

As we shall see in Section~\ref{sect:frob}, there is a geometric reason for such a symmetry. It is 
related to the fact that a certain Lie algebra is Frobenius.
\begin{thm}     \label{thm:spektr-fonin}
If $\g$ is a simple Lie algebra, then $-1\le\gamma(x_\gK)\le 2$ for all $\gamma\in\Delta^+$. More 
precisely,
\begin{enumerate}
\item \ If\/ $\g$ is of type $\GR{A}{2p-1}$ or $\GR{C}{n}$, then 
$\{\gamma(x_\gK)\mid \gamma\in\Delta^+\}=\{0,1\}$;
\item \ If\/ $\g$ is of type $\GR{A}{2p}$, then 
$\{\gamma(x_\gK)\mid \gamma\in\Delta^+\}=\{0,\frac{1}{2}, 1\}$;
\item \ For all other types, i.e., if $\theta$ is fundamental, we have 
$\{\gamma(x_\gK)\mid \gamma\in\Delta^+\}=\{-1,0,1,2\}$.
\end{enumerate}
In particular, if\/ $\g$ is not of type $\GR{A}{2p}$, then $\gamma(x_\gK)\in \BZ$ \ for any 
$\gamma\in\Delta^+$.  
\end{thm}
\begin{proof}
{\bf 1$^o$}. Using explicit formulae for $\gK$, all these assertions can be verified case-by-case. For 
instance, data of Example~\ref{ex:sl-sp} provide a proof for (1) and (2).  
However, this does not explain the general constraints $-1\le \gamma(x_\gK)\le 2$. Below we provide 
a more conceptual argument, which also uncovers some additional properties of $x_\gK$.

{\bf 2$^o$}. If $m=\rk\g=n$, then $\gK$ is a basis for $\te^*$. Hence $\gamma$ can be written as 
$\gamma=\sum_{i=1}^n k_i\beta_i$, where $k_i=\frac{1}{2}(\gamma,\beta_i^\vee)\in \frac{1}{2}\BZ$, and  
$\gamma(x_\gK)=\sum_{i=1}^n k_i$. We are to prove that $-1\le \sum_{i=1}^n k_i\le 2$. 

If $\gamma=\beta_i$, then $\gamma(x_\gK)=1$. Therefore we assume below that $\gamma\not\in\gK$.
If $\gamma\in \Delta(\h_i)$ for some $i\ge 1$, then the whole argument can be performed for the simple
Lie subalgebra $\g\lg i\rg\subset \g$ and the cascade $\gK(\g\lg i\rg)\subset \gK$, which has the unique 
maximal  element $\beta_i$. Since $\rk\g\lg i\rg< \rk\g$ for $i\ge 2$, it suffices to prove the assertion for 
$i=1$ and $\beta_1=\theta$.

Assume that $\gamma\in \eus H_1=\Delta(\h_1)$ and $\gamma\ne\theta$. Then 
$(\gamma,\theta^\vee)=1$,  $\gamma=\frac{1}{2}\theta+\sum_{i=2}^n k_i\beta_i$, and
\beq   \label{eq:star}
   4(\gamma,\gamma)=(\theta,\theta)+\sum_{i\ge 2}(2k_i)^2(\beta_i,\beta_i) .
\eeq
\indent 
{\sf (i)} \ For $\gamma\in\Delta^+_l$, it follows from~\eqref{eq:star} that
$\displaystyle 3=\sum_{i\ge 2} 4k_i^2{\cdot}\frac{(\beta_i,\beta_i)}{(\gamma,\gamma)}$. Since 
$\#(\gK\cap\Delta^+_s)\le 1$ (Lemma~\ref{lm:1-short}), the only possibilities for the nonzero coefficients $k_i$ are:
\begin{itemize}
\item \ $k_2, k_3, k_4=\pm\frac{1}{2}$ \ with \ $\beta_2,\beta_3,\beta_4\in \Delta^+_l$;
\item \ $k_2 =\pm\frac{1}{2}$, $k_3=\pm 1$ \ with \ $\beta_2\in\Delta^+_l$,  $\beta_3\in\Delta^+_s$
and $(\beta_2,\beta_2)/(\beta_3,\beta_3)=2$; \\ {}
[this happens only for $\GR{B}{2p+1}$];
\item \ $k_2 =\pm\frac{3}{2}$ \ with \ $\beta_2\in\Delta^+_s$ and $(\theta,\theta)/(\beta_2,\beta_2)=3$
\quad [this happens only for $\GR{G}{2}$].
\end{itemize}
In all these cases, we have $\gamma(x_\gK)\in \{-1,0,1,2\}$, as required.

{\sf (ii)} \ If $\gamma\in\Delta^+_s$ and $\displaystyle \frac{(\theta,\theta)}{(\gamma,\gamma)}=2$, then \eqref{eq:star} shows that
$\displaystyle  2=\sum_{i\ge 2} 4k_i^2{\cdot}\frac{(\beta_i,\beta_i)}{(\gamma,\gamma)}\ge 
\sum_{i\ge 2} 4k_i^2$. Since $\#(\gK\cap\Delta^+_s)\le 1$, the only possibility here is: 
\begin{itemize}
\item $k_2 =\pm\frac{1}{2}$ \ with \ $\beta_2\in\Delta^+_l$ \quad
[this happens for $\GR{B}{n}$, $\GR{C}{n}$,  and $\GR{F}{4}$].
\end{itemize}
Therefore $\gamma(x_\gK)=\frac{1}{2}+k_2\in \{0,1\}$.

{\sf (iii)} \ If $\gamma\in\Delta^+_s$ and $\displaystyle \frac{(\theta,\theta)}{(\gamma,\gamma)}=3$, then
\eqref{eq:star} shows that 
$\displaystyle  1=\sum_{i\ge 2} 4k_i^2{\cdot}\frac{(\beta_i,\beta_i)}{(\gamma,\gamma)}$. Here the only
possibility is $k_2=\pm \frac{1}{2}$ and $\beta_2$ is short \quad [this happens for $\GR{G}{2}$].

{\bf 3$^o$}. If $m<\rk\g=n$, then $\gK$ is not a basis for $\te^*$. Nevertheless, one can circumvent this
obstacle as follows. Let $\omega_0\in W$ be the longest element. Then $-\omega_0\in GL(\te)$ takes 
$\Delta^+$ to itself and $\beta_i\in \te^{-\omega_0}$ for each $i$, see Prop.~\ref{prop:longest&K}. Moreover, $\gK$ is a basis for 
$\te^{-\omega_0}$, see~\cite[Lemma\,6.2]{p05}. Hence 
$\bar\gamma:=\frac{1}{2}(\gamma-\omega_0(\gamma))=\sum_{i=1}^m k_i\beta_i$, \ $k_i\in \frac{1}{2}\BZ$,  and
\[
   \sum_{i=1}^m k_i=\bar\gamma(x_\gK)=\gamma(x_\gK) .
\]
Therefore, the argument of part {\bf 2$^o$} applies to $\bar\gamma$ in place of $\gamma$. However, a 
new phenomenon may occur here. As above, we begin with $\gamma\in \Delta(\h_1)\setminus \theta$. 
Then $\bar\gamma=\frac{1}{2}\theta +\sum_{i\ge 2}k_i \beta_i$. But in this case, $\bar\gamma$ is not 
necessarily a root and it may happen that $k_i=0$ for $i\ge 2$. Then $\bar\gamma(x_\gK)=1/2$.
(This does occur for $\g$ of type $\GR{A}{2p}$: if $\gamma=\esi_1-\esi_{p+1}$, then
$\omega_0(\gamma)=\esi_{2p+1}-\esi_{p+1}$ and $\bar\gamma=\frac{1}{2}\theta$. Conversely, if 
$\bar\gamma=\frac{1}{2}\theta$, then $\htt(\theta)=2{\cdot}\htt(\gamma)$. Hence the Coxeter number of 
$\g$ is odd, and this happens only for $\GR{A}{2p}$.)

{\bf 4$^o$}. Recall that $\omega_0=-1$ if and only $m=\rk\g$ and then $\bar\gamma=\gamma$. 
Therefore, part~{\bf 2$^o$} can be thought of as a special case of a more general approach outlined in 
{\bf 3$^o$}. 
\end{proof}

\begin{rmk}   \label{rem:discussion}
An analysis of possibilities for $\{k_i\}$ in the proof of Theorem~\ref{thm:spektr-fonin} reveals the following features:
\begin{itemize}
\item[\bf (1)] \ for any $\gamma\in \Delta^+$, $\bar\gamma$ is a linear combination of at most four different elements of $\gK$;
\item[\bf (2)] \ if $\#\gK=\rk\g$ and $\gK\subset \Delta^+_l$, then every $\gamma\in\Delta^+_l\setminus \gK$ is presented as a linear combination of exactly four different elements of $\gK$; 
\item[\bf (3)] \ if $\gamma(x_\gK)=2$ or $\gamma(x_\gK)=-1$, then $\gamma\in \Delta^+_l$. 
\end{itemize}
\end{rmk}

\begin{ex}    \label{ex:so-x_K}
Let $\g=\mathfrak{so}_N$ be realised as the set of skew-symmetric $N\times N$-matrices w.r.t. the antidiagonal. 
\begin{itemize}
\item \ For $N=2n$, one has $\te=\{\mathsf{diag}(x_1,\dots,x_n,-x_n,\dots,-x_1)\mid x_i\in \BC\}$. 
\begin{itemize}
\item If $n=2k$, then $\#\gK=2k$ and the entries of $x_\gK$ are $x_{2i-1}=1$
and $x_{2i}=0$ for $i=1,\dots, k$.
\item If $n=2k+1$, then still $\#\gK=2k$, the entries $x_j$ with $j\le 2k$ are the same as for 
$\mathfrak{so}_{4k}$, and $x_{2k+1}=0$. 
\end{itemize}
\item For $N=2n+1$, one has $\te=\{\mathsf{diag}(x_1,\dots,x_n,0,-x_n,\dots,-x_1)\mid x_i\in \BC\}$ 
and $\#\gK=n$. Here the entries of $x_\gK$ are $x_{2i-1}=1$ for $i\le [(n+1)/2]$ and $x_{2i}=0$ for $i\le [n/2]$.
\end{itemize}
\end{ex}
\noindent
Using Examples~\ref{ex:sl-sp} and \ref{ex:so-x_K}, one readily computes the numbers
$\{\ap(x_\gK)\}$ with $\ap\in\Pi$ for all classical Lie algebras. For the exceptional Lie algebras one
can use explicit formulae for $\gK$, see Appendix~\ref{sect:tables}. An alternative approach is to use the
recursive construction of $\gK$. One has $\Pi=\bigsqcup_{i=1}^m \Phi(\beta_i)$, and it suffices to 
describe the numbers $\ap(x_\gK)$ for  any simple Lie algebra and $\ap\in\Phi(\beta_1)=\Phi(\theta)$.
\begin{enumerate}
\item If $\theta$ is fundamental, then $\Phi(\theta)=\{\tap\}\subset\Pi_l$ and $\tap(x_\gK)=-1$
(Lemma~\ref{lm:not-dominant});
\item if $\g=\spn$, then $\theta=2\varpi_1$ and $\Phi(\theta)=\{\ap_1\}\subset\Pi_s$. Here
$\theta=2\ap_1+\beta_2$ and $\ap_1(x_\gK)=0$;
\item For $\slno$ ($n\ge 2$), we have $\Phi(\theta)=\{\ap_1,\ap_n\}$. If $n\ge 3$, then
$\ap_1(x_\gK)=\ap_n(x_\gK)=0$; if $n=2$, then $\theta=\ap_1+\ap_2$ and 
$\ap_1(x_\gK)=\ap_2(x_\gK)=1/2$;
\item For $\tri$, one has $\theta=\ap_1$ and $\ap_1(x_\gK)=1$.
\end{enumerate}
The resulting labelled diagrams are presented in Figures~\ref{fig:class} and~\ref{fig:except}.

\begin{center}
\begin{figure}[ht]   
\begin{tabular}{rcr}
$\GR{A}{2p}$\ :  & & \rule{0pt}{3ex} 
\raisebox{-1.25ex}{\begin{tikzpicture}[scale=0.75, transform shape]  
\draw (0,0.2) node[above] {\small $0$} 
        (2,0.2) node[above] {\small $0$}
        (3,0.2) node[above] {\small $1/2$} 
        (4,0.2) node[above] {\small $1/2$}
        (5,0.2) node[above] {\small $0$}
        (7,0.2) node[above] {\small $0$};
\tikzstyle{every node}=[circle, draw, fill=yellow!50]
\node (a) at (0,0) {};
\node (c) at (2,0) {};
\node (d) at (3,0) {};
\node (e)  at (4,0) {};
\node (f) at (5,0) {};
\node (h) at (7,0) {};
\tikzstyle{every node}=[circle]
\node (b) at (1,0) {$\cdots$};
\node (g) at (6,0) {$\cdots$};
\foreach \from/\to in {a/b, b/c, c/d, d/e, e/f, f/g, g/h}  \draw[-] (\from) -- (\to);   
\end{tikzpicture} }
\\
$\GR{A}{2p-1}$\ :  & & \rule{0pt}{3ex} 
\raisebox{-1.25ex}{\begin{tikzpicture}[scale=0.75, transform shape]  
\draw (0,0.2) node[above] {\small $0$} 
        (2,0.2) node[above] {\small $0$}
        (3,0.2) node[above] {\small $1$} 
        (4,0.2) node[above] {\small $0$}
        (6,0.2) node[above] {\small $0$};
\tikzstyle{every node}=[circle, draw, fill=yellow!50]
\node (a) at (0,0) {};
\node (c) at (2,0) {};
\node (d) at (3,0) {};
\node (e)  at (4,0) {};
\node (g) at (6,0) {};
\tikzstyle{every node}=[circle]
\node (b) at (1,0) {$\cdots$};
\node (f) at (5,0) {$\cdots$};
\foreach \from/\to in {a/b, b/c, c/d, d/e, e/f, f/g}  \draw[-] (\from) -- (\to);   
\end{tikzpicture} }
\\
$\GR{C}{p}$\ :  & & \rule{0pt}{3.5ex} 
\raisebox{-1.25ex}{\begin{tikzpicture}[scale=0.75, transform shape]  
\draw (0,0.2) node[above] {\small $0$} 
        (1,0.2) node[above] {\small $0$}
        (2,0.2) node[above] {\small $0$} 
        (5,0.2) node[above] {\small $0$}
        (6,0.2) node[above] {\small $1$};
\tikzstyle{every node}=[circle, draw, fill=yellow!50]
\node (a) at (0,0) {};
\node (b) at (1,0) {};
\node (c) at (2,0) {};
\node (f)  at (5,0) {};
\node (g) at (6,0) {};
\tikzstyle{every node}=[circle]
\node (d) at (3,0) {};
\node (x) at (3.5,0) {$\cdots$};
\node (e) at (4,0) {};
\foreach \from/\to in {a/b, b/c, c/d, e/f}  \draw[-] (\from) -- (\to);   
\node (r) at (5.35,.01) {\large $<$};
\draw (5.3, .06) -- +(.48,0);
\draw (5.3, -.06) -- +(.48,0);
\end{tikzpicture} }
\\
$\GR{B}{2p-1}$\ :  & & \rule{0pt}{4ex} 
\raisebox{-1.3ex}{\begin{tikzpicture}[scale=0.75, transform shape]  
\draw (0,0.2) node[above] {\small $1$} 
        (.85,0.15) node[above] {\small $-1$}
        (2,0.2) node[above] {\small $1$} 
        (4.85,0.15) node[above] {\small $-1$}
        (6,0.2) node[above] {\small $1$};
\tikzstyle{every node}=[circle, draw, fill=yellow!50]
\node (a) at (0,0) {};
\node (b) at (1,0) {};
\node (c) at (2,0) {};
\node (f)  at (5,0) {};
\node (g) at (6,0) {};
\tikzstyle{every node}=[circle]
\node (d) at (3,0) {};
\node (x) at (3.5,0) {$\cdots$};
\node (e) at (4,0) {};
\foreach \from/\to in {a/b, b/c, c/d, e/f}  \draw[-] (\from) -- (\to);   
\node (r) at (5.65,.01) {\large $>$};
\draw (5.2, .06) -- +(.48,0);
\draw (5.2, -.06) -- +(.48,0);
\end{tikzpicture} }
\\
$\GR{B}{2p}$\ :  &  &  \rule{0pt}{4ex} 
\raisebox{-1.3ex}{\begin{tikzpicture}[scale=0.75, transform shape]  
\draw (0,0.2) node[above] {\small $1$} 
        (.85,0.15) node[above] {\small $-1$}
        (2,0.2) node[above] {\small $1$} 
        (4.85,0.15) node[above] {\small $-1$}
        (6,0.2) node[above] {\small $1$}
        (7,0.2) node[above] {\small $0$};
\tikzstyle{every node}=[circle, draw, fill=yellow!50]
\node (a) at (0,0) {};
\node (b) at (1,0) {};
\node (c) at (2,0) {};
\node (f)  at (5,0) {};
\node (g) at (6,0) {};
\node (h) at (7,0) {};
\tikzstyle{every node}=[circle]
\node (d) at (3,0) {};
\node (x) at (3.5,0) {$\cdots$};
\node (e) at (4,0) {};
\foreach \from/\to in {a/b, b/c, c/d, e/f, f/g}  \draw[-] (\from) -- (\to);   
\node (r) at (6.65,.01) {\large $>$};
\draw (6.2, .06) -- +(.48,0);
\draw (6.2, -.06) -- +(.48,0);
\end{tikzpicture} }
\\
$\GR{D}{2p}$\ : &  & \rule{0pt}{5ex} 
\raisebox{-2.2ex}{\begin{tikzpicture}[scale=0.75, transform shape]  
\draw (0,0.2) node[above] {\small $1$} 
        (.85,0.2) node[above] {\small $-1$}
        (2,0.2) node[above] {\small $1$} 
        (5.9,0.2) node[above] {\small $-1$}
        (7.2, -.4) node[right] {\small $1$} 
        (7.2, .6) node[right]  {\small $1$} ;
\tikzstyle{every node}=[circle, draw, fill=yellow!50]
\node (a) at (0,0) {};
\node (b) at (1,0) {};
\node (c) at (2,0) {};
\node (g) at (6,0) {};
\node (h) at (7,.5) {};
\node (i)  at (7,-.5){};
\tikzstyle{every node}=[circle]
\node (d) at (3,0) {};
\node (e) at (4,0) {$\cdots$};
\node (f)  at (5,0) {};
\foreach \from/\to in {a/b, b/c, c/d, f/g, g/h, g/i}  \draw[-] (\from) -- (\to);        
\end{tikzpicture} }
\\ 
$\GR{D}{2p+1}$\ : &  & \rule{0pt}{5ex} 
\raisebox{-2.5ex}{\begin{tikzpicture}[scale=0.75, transform shape]  
\draw (0,0.2) node[above] {\small $1$} 
        (.85,0.2) node[above] {\small $-1$}
        (2,0.2) node[above] {\small $1$} 
        (4.85,0.2) node[above] {\small $-1$}
        (6,0.2) node[above] {\small $1$}
        (7.2, -.5) node[right] {\small $0$} 
        (7.2, .5) node[right]  {\small $0$} ;
\tikzstyle{every node}=[circle, draw, fill=yellow!50]
\node (a) at (0,0) {};
\node (b) at (1,0) {};
\node (c) at (2,0) {};
\node (f)  at (5,0) {};
\node (g) at (6,0) {};
\node (h) at (7,.5) {};
\node (i)  at (7,-.5){};
\tikzstyle{every node}=[circle]
\node (d) at (3,0) {};
\node (x) at (3.5,0) {$\cdots$};
\node (e) at (4,0) {};
\foreach \from/\to in {a/b, b/c, c/d, e/f, f/g, g/h, g/i}  \draw[-] (\from) -- (\to);        
\end{tikzpicture} }
\end{tabular}
\caption{Numbers $\{\ap(x_{\eus K})\mid \ap\in\Pi\}$ for the classical cases}   \label{fig:class}   
\end{figure}
\end{center}   
\begin{center}
\begin{figure}[ht] 
\begin{tabular}{rcr}
$\GR{E}{6}$\ : & &   \rule{0pt}{6ex} 
\raisebox{-2.7ex}{\begin{tikzpicture}[scale=0.75, transform shape]  
\draw (1,0.2) node[above] {\small $0$}
        (2,0.2) node[above] {\small $0$} 
        (3,0.2) node[above] {\small $1$}
        (4,0.2) node[above] {\small $0$} 
        (5,0.2) node[above] {\small $0$}
        (3.2,-1) node[right]  {\small $-1$} ;
\tikzstyle{every node}=[circle, draw, fill=orange!50]
\node (b) at (1,0){};
\node (c) at (2,0) {};
\node (d) at (3,0) {};
\node (e) at (4,0) {};
\node (f) at (5,0) {};
\node (g) at (3,-1) {};
\foreach \from/\to in {b/c, c/d, d/e, e/f, d/g}  \draw[-] (\from) -- (\to);        
\end{tikzpicture} }
\\
$\GR{E}{7}$\ : & & \rule{0pt}{6.5ex} 
\raisebox{-2.7ex}{\begin{tikzpicture}[scale=0.75, transform shape]  
\draw (0,0.2) node[above] {\small $1$} 
        (.9,0.2) node[above] {\small $-1$}
        (2,0.2) node[above] {\small $1$} 
        (2.9,0.2) node[above] {\small $-1$}
        (4,0.2) node[above] {\small $1$} 
        (4.9,0.2) node[above] {\small $-1$}
        (3.2,-1) node[right]  {\small $1$} ;
\tikzstyle{every node}=[circle, draw, fill=orange!50]
\node (a) at (0,0) {};
\node (b) at (1,0){};
\node (c) at (2,0) {};
\node (d) at (3,0) {};
\node (e) at (4,0) {};
\node (f) at (5,0) {};
\node (g) at (3,-1) {};
\foreach \from/\to in {a/b, b/c, c/d, d/e, e/f, d/g}  \draw[-] (\from) -- (\to);        
\end{tikzpicture} }
\\
$\GR{E}{8}$\ : & & \rule{0pt}{6.5ex} 
\raisebox{-2.7ex}{\begin{tikzpicture}[scale=0.75, transform shape]  
\draw (-1.1,0.2) node[above] {\small $-1$}
        (0,0.2) node[above] {\small $1$} 
        (.9,0.2) node[above] {\small $-1$}
        (2,0.2) node[above] {\small $1$} 
        (2.9,0.2) node[above] {\small $-1$}
        (4,0.2) node[above] {\small $1$} 
        (4.9,0.2) node[above] {\small $-1$}
        (3.2,-1) node[right]  {\small $1$} ;
\tikzstyle{every node}=[circle, draw, fill=orange!50]
\node (h) at (-1,0) {};
\node (a) at (0,0) {};
\node (b) at (1,0){};
\node (c) at (2,0) {};
\node (d) at (3,0) {};
\node (e) at (4,0) {};
\node (f) at (5,0) {};
\node (g) at (3,-1) {};
\foreach \from/\to in {h/a, a/b, b/c, c/d, d/e, e/f, d/g}  \draw[-] (\from) -- (\to);        
\end{tikzpicture} } 
\\
$\GR{F}{4}$\ : & & \rule{0pt}{5ex} 
\raisebox{-1.4ex}{\begin{tikzpicture}[scale=0.75, transform shape]  
\draw (0,0.2) node[above] {\small $0$} 
        (1,0.2) node[above] {\small $0$}
        (2,0.2) node[above] {\small $1$} 
        (3,0.2) node[above] {\small $-1$};
\tikzstyle{every node}=[circle, draw, fill=orange!50]
\node (a) at (0,0) {};
\node (b) at (1,0) {};
\node (c) at (2,0) {};
\node (d)  at (3,0) {};
\foreach \from/\to in {a/b, c/d}  \draw[-] (\from) -- (\to);  
\tikzstyle{every node}=[circle]
\node (r) at (1.35,.01) {\large $<$};
\draw (1.31, .06) -- +(.47,0);
\draw (1.31, -.06) -- +(.47,0);
\end{tikzpicture} }
\\
$\GR{G}{2}$\ : & &  \rule{0pt}{5ex} 
\raisebox{-1.4ex}{\begin{tikzpicture}[scale=0.75, transform shape]  
\draw (0,0.2) node[above] {\small $1$} 
        (1,0.2) node[above] {\small $-1$};
\tikzstyle{every node}=[circle, draw, fill=orange!50]
\node (a) at (0,0) {};
\node (b) at (1,0) {};
\tikzstyle{every node}=[circle]
\node (r) at (0.35,.01) {\large $<$};
\draw (.31, .06) -- +(.47,0);
\draw (.25, 0) -- +(.5,0);
\draw (.31, -.06) -- +(.47,0);
\end{tikzpicture} }
\end{tabular}
\caption{Numbers $\{\ap(x_{\eus K})\mid \ap\in\Pi\}$ for the exceptional cases}   \label{fig:except}   
\end{figure}
\end{center}
\vskip-3ex
Let us summarise main features of the diagrams obtained.
\begin{rmk}    \label{rem:alternate&connected}
1) Fractional values occur only for $\g$ of type $\GR{A}{2p}$. For $\GR{A}{2p-1}$ and $\GR{C}{p}$, the dominant element $x_\gK$ is a multiple of the fundamental weight $\varpi_p$. More precisely,
$x_\gK=\frac{2}{(\ap_p,\ap_p)}\varpi_p$ (cf. Remark~\ref{ex:sl-sp}).

2) All these diagrams have no marks `$2$' and all the marks $\{\ap(x_{\eus K})\}$ are nonzero if and 
only if $\g$ is of type $\GR{B}{2p-1}, \GR{D}{2p}, \GR{E}{7}, \GR{E}{8}, \GR{G}{2}$. The 
subset $\{\ap\in\Pi\mid \ap(x_\gK)=\pm 1\}$ is always connected and the marks `1' and `$-1$' alternate in 
this subset.

3)  If $\ap(x_\gK)=-1$ and $\ap'\in\Pi$ is adjacent to $\ap$, then $\ap'(x_\gK)=1$. Moreover, if
$\ap'(x_\gK)=1$, then $\ap'\in\gK$.

4) The numbers $\{\ap(x_{\eus K})\}$ are compatible with the unfolding procedures
$\GR{C}{p}\mapsto \GR{A}{2p-1}$, $\GR{B}{p-1}\mapsto \GR{D}{p}$, $\GR{F}{4}\mapsto \GR{E}{6}$, 
and $\GR{G}{2}\mapsto \GR{D}{4}$.  For instance, \quad
\raisebox{-1.3ex}{\begin{tikzpicture}[scale=0.75, transform shape]  
\draw (0,0.2) node[above] {\small $1$} 
        (1,0.2) node[above] {\small $-1$};
\tikzstyle{every node}=[circle, draw, fill=orange!50]
\node (a) at (0,0) {};
\node (b) at (1,0) {};
\tikzstyle{every node}=[circle]
\node (r) at (0.35,.01) {\large $<$};
\draw (.31, .06) -- +(.47,0);
\draw (.25, 0) -- +(.5,0);
\draw (.31, -.06) -- +(.47,0);
\end{tikzpicture} } \ $\mapsto$  \ 
\raisebox{-2.5ex}{\begin{tikzpicture}[scale=0.75, transform shape]  
\draw (5,0.2) node[above] {\small $1$} 
        (5.9,0.2) node[above] {\small $-1$}
        (7.2, -.4) node[right] {\small $1$} 
        (7.2, .6) node[right]  {\small $1$} ;
\tikzstyle{every node}=[circle, draw, fill=yellow!50]
\node (f)  at (5,0) {};
\node (g) at (6,0) {};
\node (h) at (7,.5) {};
\node (i)  at (7,-.5){};
\foreach \from/\to in {f/g, g/h, g/i}  \draw[-] (\from) -- (\to);        
\end{tikzpicture} }.
\end{rmk}

\section{The cascade element and self-dual representations of $\g$}
\label{sect:self-dual}

Consider the standard lattices in $\te^*_\BQ$ associated with $\Delta$~\cite[Chap.\,4,\,\S2.8]{VO}:
\begin{itemize}
\item \ $\mathcal Q=\bigoplus_{i=1}^n \BZ\ap_i$ -- the root lattice;
\item \ $\mathcal Q^\vee=\bigoplus_{i=1}^n \BZ\ap_i^\vee$ -- the coroot lattice;
\item \ $\mathcal P=\bigoplus_{i=1}^n \BZ\varpi_i$ -- the weight lattice;
\item \ $\mathcal P^\vee=\bigoplus_{i=1}^n \BZ\varpi_i^\vee$ -- the coweight lattice, where 
$\varpi_i^\vee=2\varpi_i/(\ap_i,\ap_i)$.
\end{itemize}
Then $\mathcal P\supset\mathcal Q$, $\mathcal P^\vee\supset\mathcal Q^\vee$,
$\mathcal P=(\mathcal Q^\vee)^*$, and $\mathcal P^\vee=\mathcal Q^*$, where $\mathcal L^*$ stands for the dual lattice of $\mathcal L$. For instance,
$\mathcal P^\vee=\mathcal Q^*=\{\nu\in \te_\BQ\mid (\nu,\gamma)\in \BZ \ \ \forall \gamma\in\Delta\}$.

If $\g$ is not of type $\GR{A}{2p}$, then $x_\gK\in \mathcal P^\vee$ (Theorem~\ref{thm:spektr-fonin}). 
However, then $x_\gK$ does not always belong to $\mathcal Q^\vee$, and we characterise below the relevant cases.
If $\eus M\subset \te_\BQ$ is finite, 
then $|\eus M|:=\sum_{m\in\eus M} m$. As usual, set
$2\varrho=\sum_{\gamma\in \Delta^+}\gamma=|\Delta^+|$ and 
$2\varrho^\vee=\sum_{\gamma\in \Delta^+}\gamma^\vee=|(\Delta^\vee)^+|$.
Then $\mathsf h(\g)=(\varrho^\vee, \theta)+1$ and the {\it dual Coxeter number\/} of $\g$ is
$\mathsf h^*=\mathsf h^*(\g):=(\varrho, \theta^\vee)+1$.

\begin{lm}    \label{lm:rho-vee}  
One has \ $2\varrho=\sum_{i=1}^m \bigl(\mathsf h^*(\g\lg j\rg)-1\bigr)\beta_j$ \ and \ 
$2\varrho^\vee=\sum_{i=1}^m \bigl(\mathsf h(\g\lg j\rg)-1\bigr)\beta_j^\vee$.
\end{lm}
\begin{proof}
Since $\Delta^+=\bigsqcup_{i=1}^m \gH_{\beta_i}$ and $\beta_1=\theta$,  it is sufficient to prove that 
$|\gH_\theta^\vee|=\bigl(\mathsf h(\g)-1\bigr)\theta^\vee$ and 
$|\gH_\theta|=\bigl(\mathsf h^*(\g)-1\bigr)\theta$. 
Since $\gH_\theta\setminus \{\theta\}$ is the union of pairs $\{\gamma, \theta-\gamma\}$, where the roots
$\gamma$ and $\theta-\gamma$ have the same length, it is clear 
that $|\gH_\theta^\vee|= a\theta^\vee$ for some $a\in\BN$. Then 
\[
   a=(\textstyle \frac{1}{2}|\gH_\theta^\vee|, \theta)=(\varrho^\vee, \theta)=\mathsf h(\g)-1. 
\]
The proof of the second relation is similar.
\end{proof}

\begin{prop}    \label{prop:self-dual}
The following conditions are equivalent:
\begin{enumerate}
\item $x_\gK\in \mathcal Q^\vee$;
\item every self-dual representation of $\g$ is orthogonal.
\end{enumerate}
\end{prop}
\begin{proof}
By a classical result of Dynkin, if $\pi_\lb: \g\to \BV_\lb$ is an irreducible representation with highest 
weight $\lb$, then $\pi_\lb$ is self-dual if and only if $\omega_0(\lb)=-\lb$. Then $\pi_\lb$ is orthogonal 
if and only if $(\varrho^\vee, \lb)\in \BN$~\cite{dy50}, cf. also~\cite[Exercises\,4.2.12--13]{VO} or
\cite[Chap.\,3, \S\,2.7]{t41}. Hence every self-dual 
representation of $\g$ is orthogonal if and only if $\varrho^\vee\in\mathcal Q^\vee$.

Therefore, it suffices to prove that $\varrho^\vee-x_\gK\in\mathcal Q^\vee$, if $\g$ is not of type 
$\GR{A}{2p}$. By Lemma~\ref{lm:rho-vee}, we have
\[
   \varrho^\vee-x_\gK=\sum_{i=1}^m \frac{\mathsf h(\g\lg j\rg)-2}{2}{\cdot}\beta_j^\vee .
\]
It remains to observe that, for any simple $\g$, the Coxeter numbers $\mathsf h(\g\lg j\rg)$, 
$j=1,\dots,m$, have the same parity, and if $\g$ is not of type $\GR{A}{2p}$, then all these numbers 
are even.
\end{proof}

Using Proposition~\ref{prop:self-dual} and~\cite[Table\,3]{VO}, one readily obtains that
\\[.6ex]
\centerline{
$x_\gK\in \mathcal Q^\vee \ \Longleftrightarrow \ \g\in\{\GR{B}{4p-1}, \GR{B}{4p}, \GR{D}{4p}, 
\GR{D}{4p+1}, \GR{E}{6}, \GR{E}{8}, \GR{F}{4}, \GR{G}{2}\}$ .
}
\\[.6ex]
Although the coefficients $[\varrho^\vee:\beta_i^\vee]$ are usually not integral, this does {\bf not} 
necessarily mean that $\varrho^\vee\not\in \mathcal Q^\vee$. For, the elements 
$\beta_1^\vee,\dots,\beta_m^\vee$ do not form a (part of a) basis for $\mathcal Q^\vee$.

\section{The cascade element as the Ooms element of a Frobenius Lie algebra}
\label{sect:frob}

\noindent
Given an arbitrary Lie algebra $\q$, one associates the {\it Kirillov form\/} $\gB_\eta$ on $\q$ to any 
$\eta\in\q^*$. By definition, if $\lg\ ,\, \rg: \q^*\times\q\to \BC$ is the natural pairing and
$x,y\in\q$, then
\[
    \gB_\eta(x,y)=\lg \eta, [x,y]\rg=-\lg \ads(x)\eta,y\rg .
\]
Then $\gB_\eta$ is skew-symmetric and $\Ker \gB_\eta=\q^\eta$, the stabiliser of $\eta$ in $\q$. The 
{\it index\/} of $\q$ is $\ind\q=\min_{\eta\in\g^*}\dim\q^\eta$, and 
$\q^*_{\sf reg}=\{\xi\in \q^*\mid \dim\q^\xi=\ind\q\}$ is the set of {\it regular elements\/} of $\q^*$ .
Suppose that $\q$ is {\it Frobenius\/}, i.e., there is $\xi\in\q^*$ such that $\gB_\xi$ is non-degenerate. 
Then $\q^\xi=\{0\}$ and $\xi\in\q^*_{\sf reg}$. The $2$-form $\gB_\xi$ yields a linear isomorphism 
between $\q$ and $\q^*$. Let 
$x_{\q,\xi}=x_\xi\in\q$ correspond to $\xi$ under that isomorphism. It is noticed by A.\,Ooms~\cite{ooms} 
that $\ad x_\xi$ enjoys rather interesting properties. Namely, using the non-degenerate form $\gB_\xi$, one 
defines the adjoint operator $(\ad x_\xi)^*: \q\to \q$. By~\cite[Theorem\,3.3]{ooms}, one has
\beq        \label{eq:ooms}
       (\ad x_\xi)^*=1- \ad x_\xi .
\eeq
Therefore, if $\lb$ is an eigenvalue of $\ad x_\xi$ with multiplicity $m_\lb$, then $1-\lb$ is also an 
eigenvalue and $m_\lb=m_{1-\lb}$. Hence $\tr_\q(\ad x_\xi)=(\dim\q)/2$.
We say that $x_\xi$ is the {\it Ooms element\/} associated with 
$\xi\in\q^*_{\sf reg}$. 
Another way to define $x_\xi$ is as follows. Since $\q^\xi=\{0\}$, we have $\q{\cdot}\xi=\q^*$. Then 
$x_\xi\in\q$ is the unique element such that $(\ads x_\xi){\cdot}\xi=-\xi$. 

If $\q=\Lie(Q)$ is algebraic, then each element of $\q$ has the Jordan 
decomposition~\cite[Ch.\,3. \S\,3.7]{VO}. Furthermore, if $\q$ is Frobenius and algebraic, then
$\q^*_{\sf reg}$ is the dense $Q$-orbit in $\q^*_{\sf reg}$. Therefore, all Ooms elements in $\q$ are 
$Q$-conjugate, and we can also write $x_\xi=x_\q$ for an Ooms element in $\q$. 
\begin{lm}   
\label{lm:ooms-ss}
If\/ $\q$ is Frobenius and algebraic, then any Ooms element $x_\xi$ is semisimple.
\end{lm}
\begin{proof}
For the Jordan decomposition $x_\xi=(x_\xi)_s+(x_\xi)_n$, the defining relation 
$(\ads x_\xi){\cdot}\xi=-\xi$ obviously implies that $\ads ((x_\xi)_s){\cdot}\xi=-\xi$. Then the uniqueness of 
the Ooms element associated with $\xi$ shows that $x_\xi=(x_\xi)_s$ is semisimple. 
\end{proof}
Let $\spec (x_\xi)$ denote the multiset of eigenvalues of $\ad x_\xi$ in $\q$. In other 
words, if $\lb$ is an eigenvalue and $\q(\lb)$ is the corresponding eigenspace, then 
$\lb\in \spec (x_\xi)$ is taken with multiplicity $\dim\q(\lb)$. Then it follows from \eqref{eq:ooms} 
that $\spec (x_\xi)$ is symmetric w.r.t. $1/2$.

In~\cite{p22}, we defined the {\it Frobenius envelope\/} of the nilradical $\p^{\sf nil}$ of any standard 
parabolic subalgebra $\p\subset \g$. For $\ut=\be^{\sf nil}$, this goes as follows. Set 
$\te_\gK=\lg \beta_1,\dots,\beta_m\rg_\BC=\lg\gK\rg_\BC\subset\te$. In more direct terms, 
$\te_\gK=\bigoplus_{i=1}^m [e_{\beta_i}, e_{-\beta_i}]$. Then $\te_\gK$ is an algebraic subalgebra of 
$\te$, and the Frobenius envelope of $\ut$ is $\be_\gK=\te_\gK\oplus \ut$, which is 
an ideal of $\be$. Note that $\be_\gK=\be$ if and only if $\#\gK=\rk\g$. By~\cite[Proposition\,5.1]{p22}, 
we have $\ind\be_\gK=0$, i.e., $\be_\gK$ is Frobenius. Here $\be^*_\gK\simeq \g/\be_\gK^\perp$ can 
be identified with $\be^-_\gK=\te_\gK\oplus\ut^-$ as vector space and $\te$-module. Furthermore, under this identification, we have 
\[
    \xi_\gK=\sum_{\beta\in\gK}e_{-\beta}\in (\be_\gK)^*_{\sf reg}.
\]
Therefore, $(\ads x_\gK)\xi_\gK=-\xi_\gK$, i.e., 
the  Ooms element $x_{\xi_\gK}$ associated with $\xi_\gK$ is nothing but the cascade 
element $x_\gK$ from Section~\ref{sect:comb-prop}. Hence $\spec (x_\gK)$
is symmetric w.r.t. $1/2$. Note that since $\te_\gK\subset \be_\gK(0)$, 
$\bigoplus_{\beta\in\gK}\g_\beta\subset  \be_\gK(1)$, and 
$\dim\te_\gK=\dim (\sum_{\beta\in\gK}\g_\beta)=\#\gK$, the symmetry of the multiset
$\spec (x_\gK)$ w.r.t. $1/2$ is equivalent to the symmetry established in Lemma~\ref{lm:symmetry-value}.

In this situation, we have
\[
   \frac{1}{2}(\dim\ut +\#\gK)= \frac{1}{2}\dim\be_\gK=
   \tr_{\be_\gK}(\ad x_\gK)=\sum_{\gamma>0}\gamma(x_\gK)=2\varrho(x_\gK) .
\]
Since $\ind\ut=\#\gK=m$~\cite{jos77}, the sum $ \frac{1}{2}(\dim\ut +\#\gK)$ is the {\it magic number} 
associated with $\ut$. Comparing this with Lemma~\ref{lm:rho-vee}, we obtain
\[
     \frac{1}{2}(\dim\ut +m)=2\varrho(x_\gK)=\sum_{i=1}^m ((\mathsf h^*(\g\lg j\rg)-1) .
\]
\begin{rmk}
The case of $\g=\mathfrak{sl}_{2n+1}$ in Theorem~\ref{thm:spektr-fonin} shows that the eigenvalues of 
the Ooms element for the Frobenius algebra $\be_\gK$ are not always integral. Nevertheless, there are 
interesting classes of Frobenius algebras $\q$ such that $\spec (x_{\q})\subset \BZ$. Using meander 
graphs of type $\GR{A}{n}$~\cite{dk00} or $\GR{C}{n}$~\cite{py17}, I can explicitly describe the Ooms 
element $x_\p$ for any Frobenius {\it seaweed subalgebra\/} $\p$ of $\slno$ or $\spn$ and then prove 
that the eigenvalues of $\ad x_{\p}$ belong to $\BZ$. However, being symmetric with respect to 
$1/2$ and integral, the eigenvalues of such $x_\p$ do not always confine to the interval $[-1,2]$.
\end{rmk}

\section{The abelian ideal of $\be$ associated with the cascade}
\label{sect:ab-ideal}

If $\g$ is not of type $\GR{A}{2p}$, then $x_\gK\in\te$ has the property that 
$\{\gamma(x_\gK)\mid \gamma\in\Delta^+\}\subset \{-1,0,1,2\}$ (Theorem~\ref{thm:spektr-fonin}). By 
Kostant's extension of Peterson's theory~\cite[Section\,3]{ko98}, every such element of $\te$ determines 
an abelian ideal of $\be$. In particular, one may associate an abelian ideal of $\be$ to $x_\gK$ (i.e., to 
$\gK$) as long as $\g$ is not of type $\GR{A}{2p}$. Our goal is to characterise this ideal. If $\ah$ is an 
abelian ideal of $\be$, then $\ah$ is $\te$-stable and $\ah\subset \ut^+$. Hence it suffices to determine 
the set of positive roots $\Delta(\ah)\subset\Delta^+$.

Recall that the {\it inversion set\/} of $w\in W$ is 
$\eus N(w)=\{\gamma\in\Delta^+\mid w(\gamma)\in \Delta^-\}$.
Then $\eus N(w)$ and $\Delta^+\setminus \eus N(w)$ are {\it closed}\/ (under root addition), i.e., if
$\gamma',\gamma''\in \eus N(w)$ and $\gamma'+\gamma''\in\Delta^+$, then
$\gamma'+\gamma''\in \eus N(w)$, and likewise for $\Delta^+\setminus \eus N(w)$. Conversely, 
if $S\subset \Delta^+$ and both $S$ and $\Delta^+\setminus S$ are closed, then $S=\eus N(w)$ for
a unique $w\in W$~\cite[Prop.\,5.10]{ko61}. Below, $\gamma>0$ (resp. $\gamma<0$) is a shorthand for $\gamma\in\Delta^+$
(resp. $\gamma\in\Delta^-$).

{\it\bfseries Kostant's construction}.  
Set $\De_{\sf ab}=\{t\in \te_\BQ\mid -1\le \gamma(t)\le 2\ \ \forall \gamma\in\Delta^+\}$. Kostant 
associates the abelian ideal $\ah_z{\lhd}\be$ to each 
$z\in\De_{\sf ab}\cap\mathcal P^\vee$~\cite[Theorem\,3.2]{ko98}, i.e., if $\gamma(z)\in\{-1,0,1,2\}$ 
for any $\gamma\in\Delta^+$.
Unlike the Peterson method, his construction exploits only $W$ and does not invoke the affine Weyl 
group $\widehat W$ and "minuscule" elements in it. Set 
$\Delta^\pm_z(i)=\{\gamma\in\Delta^\pm\mid \gamma(z)=i\}$ and 
$\Delta_z(i)=\Delta^+_z(i)\cup\Delta^-_z(i)$. Note that \ $-\Delta^+_z(i)=\Delta^-_z(-i)$. We have
\beq       \label{eq:main-partition}
        \Delta^+=\bigsqcup_{i=-1}^2\Delta^+_z(i) 
\eeq
and both subsets $\Delta^+_z(1)\sqcup \Delta^+_z(2)$ and $\Delta^+_z(-1)\sqcup\Delta^+_z(0)$ are 
closed. Therefore, there is a unique $w_z\in W$ such that $\eus N(w_z)=\Delta^+_z(1)\sqcup \Delta^+_z(2)$. By definition, then
\[    
   \Delta(\ah_z)=w_z\bigl(\Delta^+_z(-1)\bigr)\sqcup w_z\bigl(-\Delta^+_z(2)\bigr) \subset \Delta^+.
\]
Hence $\dim\ah_z=\#\bigl(\Delta^+_z(-1)\cup \Delta^+_z(2)\bigr)$. Note that $\Delta_z(0)$ is a root
system in its own right, and $\Delta^+_z(0)$ is a set of positive roots in it. We say that the union 
in~\eqref{eq:main-partition} is the $z$-{\it grading}\/ of $\Delta^+$, and if $\gamma\in \Delta^+_z(i)$, 
then $i$ is the $z$-{\it degree}\/ of $\gamma$.

\begin{prop}           \label{lm:good-prop-w_K}
Let $z\in \De_{\sf ab}\cap\mathcal P^\vee$ be arbitrary. 
\begin{itemize}  
\item[\sf (i)] \ If\/ $\gamma\in \Delta^+_z(0)$ is not a sum of two roots from $\Delta^+_z(0)$, then $w_z(\gamma)\in \Pi$;
\item[\sf (ii)] \ if\/ $\theta(z)=1$, i.e., $\theta\in \Delta^+_z(1)$, then $w_z(\theta)\in -\Pi$.
\end{itemize} 
\end{prop}
\begin{proof}
Recall that $\eus N(w_z)=\Delta^+_z(1)\cup\Delta^+_z(2)$.
\\ \indent
{{\sf (i)} \ Since $\gamma\not\in \eus N(w_z)$, we have $w_z(\gamma)>0$. Assume that 
$w_z(\gamma)=\mu_1+\mu_2$, where $\mu_i>0$. Then $\gamma=w_z^{-1}(\mu_1)+
w_z^{-1}(\mu_2)$. Letting $\gamma_i:=w_z^{-1}(\mu_i)$, we get the following possibilities.

1$^o$. If $\gamma_1,\gamma_2>0$, then the $z$-grading of $\Delta^+$ shows that there are further 
possibilities:
\begin{itemize}
\item $\gamma_1\in\Delta^+_z(-1)$ and $\gamma_2\in\Delta^+_z(1)$. But then $\mu_2=w_z(\gamma_2)<0$. A contradiction!
\item $\gamma_1,\gamma_2 \in\Delta^+_z(0)$ -- this contradicts the hypothesis on $\gamma$.
\end{itemize}

2$^o$. Suppose that $\gamma_1>0$ and $\gamma_2<0$. Then 
$\gamma_1\not\in\eus N(w_z)$ and $-\gamma_2\in \eus N(w_z)$. 
\\
Hence $\gamma_1\in \Delta^+_z({\le}0)$ and $-\gamma_2\in \Delta^+_z({\ge}1)$, i.e., 
$\gamma_2\in\Delta^-_z({\le}-1)$.
It follows that $\gamma_1+\gamma_2=\gamma\in \Delta_z({\le}-1)$. A contradiction!

Thus, cases~1$^o$ and 2$^o$ are impossible, and $w_z(\gamma)$ must be simple.

\sf (ii)} \ Since $\theta(z)=1$, we have $w_z(\theta)<0$. Assume that $w_z(\theta)=-\gamma_1-\gamma_2$, where $\gamma_i>0$. Then $\theta=w_z^{-1}(-\gamma_1)+w_z^{-1}(-\gamma_2)$ and
$\mu_i:=w_z^{-1}(-\gamma_i)>0$ for $i=1,2$. Then $\mu_i\in \eus N(w_z)$ and hence
$\mu_i(z)\ge 1$. Therefore $\theta(z)=\mu_1(z)+\mu_2(z)\ge 2$. A contradiction!
\end{proof}

\begin{rmk}    \label{rem:kak-stroit-nerazlozhimye}
Note that $\gamma\in \Delta^+_z(0)$ is not a sum of two roots from $\Delta^+_z(0)$ if and only if 
$\gamma$ belongs to the {\it base} (=\,set of simple roots) of $\Delta_z(0)$ that is contained in 
$\Delta_z^+(0)$. 
\end{rmk}
Let $\gC\subset\te_\BQ$ denote the {\it dominant Weyl chamber}, i.e., $\gC=\{x\in\te_\BQ\mid
\ap(x)\ge 0 \ \ \forall \ap\in\Pi\}$.
\begin{prop}     \label{prop:anti-dom}
We have $w_z(z)\in -\gC$, i.e., it is anti-dominant. Moreover, $w_z$ is the unique element of
minimal lenght in\/ $W$ that takes $z$ into $-\gC$.
\end{prop}
\begin{proof}
For any $\lb\in\te^*_\BQ$, let $\lb^+$ be the {\sl dominant\/} representative in $W{\cdot}\lb$. 
By~\cite[Lemma\,4.1]{ion04}, there is a unique element of minimal length $w_\lb\in W$ such that 
$w_\lb{\cdot}\lb=\lb^+$ and then $\eus N(w_\lb)=\{\gamma\in\Delta^+\mid (\gamma,\lb)<0\}$ (cf. also 
\cite[Theorem\,4.1]{p03}). Translating this into the assertion on the {\sl anti-dominant\/} representative 
in $W{\cdot}z\subset\te_\BQ$, we see that the element of minimal length $\bar w$ that takes 
$z$ into $-\gC$ is defined the property that
\[
   \eus N(\bar w)=\{\gamma\in\Delta^+\mid (\gamma, z)>0\}=\Delta^+_z(1)\cup\Delta^+_z(2).
\]
Therefore, $\bar w=w_z$.
\end{proof}

We will use Kostant's construction with $z=x_\gK$. Therefore, it is assumed below that 
$\spec (x_\gK)\subset \BZ$, which excludes the series  $\GR{A}{2p}$. For simplicity, write 
$\Delta^+_\gK(i)$, $\ah_\gK$, and 
$w_\gK$ in place of $\Delta^+_{x_\gK}(i)$, $\ah_{x_\gK}$, and $w_{x_\gK}$, respectively. By the symmetry of $\spec (x_\gK)$, one has  
$\dim\ah_\gK=\#\Delta^+_\gK(-1)+\#\Delta^+_\gK(2)=2\#\Delta^+_\gK(2)$ and 
$\#\Delta^+_\gK(0)+\#\gK=\#\Delta^+_\gK(1)$. Set $\Pi_\gK(i)=\Pi\cap\Delta^+_\gK(i)$.

\begin{ex}  \label{ex:ideal-sl-sp}
For $\sltn$ and $\spn$, the element $x_\gK$ is dominant and $\Delta^+_\gK(2)=\varnothing$. 
Therefore, $\ah_\gK=\{0\}$ in these cases. In the other cases, i.e., when $\theta$ is fundamental,
$\ah_\gK$ is a non-trivial abelian ideal of $\be$.
Anyway, for {\bf all} simple Lie algebras except $\mathfrak{sl}_{2n+1}$, we obtain a non-trivial 
element $w_\gK\in W$, which possesses some interesting properties, see below.
\end{ex}

\begin{prop}       \label{prop:fundam}
The rank of the root system $\Delta_\gK(0)$ equals $\rk\g-1$ and the
dominant weight $-w_\gK(x_\gK)$ is a multiple of a fundamental weight. Namely, if 
$w_\gK(\theta)=-\ap_j\in -\Pi$ (cf. Proposition~\ref{lm:good-prop-w_K}), then 
$-w_\gK(x_\gK)=\frac{2}{(\ap_j,\ap_j)}\varpi_j=:\varpi_j^\vee$.
\end{prop}
\begin{proof}
Clearly, $-w_\gK(x_\gK)$ is a multiple of a fundamental weight if and only if the rank of the root system 
$\Delta_{\gK}(0)=\{\gamma\in\Delta\mid \gamma(x_\gK)=0\}$ equals $\rk\g-1$. Therefore, it suffices to 
point out $\rk\g-1$ linearly independent roots in $\Delta^+_\gK(0)$.

Since $\Pi_{\pm 1}:=\Pi_\gK(-1)\cup\Pi_\gK(1)$ is connected and the roots from $\Pi_\gK(-1)$ and 
$\Pi_\gK(1)$ alternate in the Dynkin diagram (see Remark~\ref{rem:alternate&connected}(2)), there 
are $(\#\Pi_{\pm 1})-1$ edges therein and each edge gives rise to the root 
$\ap_{i_1}+\ap_{i_2}\in \Delta^+_\gK(0)$. Together with the roots in $\Pi_\gK(0)$, this yields exactly 
$\rk\g-1$ linearly independent roots in $\Delta^+_\gK(0)$. Actually, these roots form the base of 
$\Delta_\gK(0)$ in $\Delta^+_\gK(0)$.

If $w_\gK(x_\gK)=-a_i\varpi_i$ and  $w_\gK(\theta)=-\ap_j$, then
\[
    1=\theta(x_\gK)= w_\gK(\theta)(w_\gK(x_\gK))=a_i(\ap_j, \varpi_i) .
\]
Hence $i=j$ and $a_i=2/(\ap_i,\ap_i)$.
\end{proof}

Important characteristics of the abelian ideal $\ah_\gK$ can be expressed via $w_\gK\in W$. Since 
$\Delta(\ah_\gK)=:\Delta_{\lg\gK\rg}$ is an upper ideal of the poset $(\Delta^+,\curle)$, it is completely 
determined by its subset $\min(\Delta_{\lg\gK\rg})$ of {\it minimal elements\/} w.r.t. the root order 
``$\curle$''. Similarly, the complement $\ov{\Delta_{\lg\gK\rg}}=\Delta^+\setminus \Delta_{\lg\gK\rg}$ is 
determined by the subset of its {\it maximal elements}, $\max(\ov{\Delta_{\lg\gK\rg}})$.

It follows from~\eqref{eq:main-partition} and the definition of $\eus N(w_\gK)$ that
\[ 
   \Delta^+=w_\gK\bigl(\Delta^+_\gK(-1)\bigr)\sqcup w_\gK\bigl(\Delta^+_\gK(0)\bigr)\sqcup 
   w_\gK\bigl(-\Delta^+_\gK(1)\bigr) \sqcup w_\gK\bigl(-\Delta^+_\gK(2)\bigr) .   
\] 
In this union, the first and last sets form $\Delta_{\lg\gK\rg}$, and two sets in the middle form $\ov{\Delta_{\lg\gK\rg}}$.
Hence 
\begin{gather}   \label{eq:delta}
   w_\gK^{-1}(\Delta_{\lg\gK\rg}) = \Delta^+_\gK(-1)\sqcup -\Delta^+_\gK(2)= 
   \Delta^+_\gK(-1)\sqcup\Delta^-_\gK(-2) \ ;
\\          \label{eq:bar-delta}
  w_\gK^{-1}(\ov{\Delta_{\lg\gK\rg}}) = \Delta^+_\gK(0)\sqcup -\Delta^+_\gK(1)= 
  \Delta^+_\gK(0)\sqcup \Delta^-_\gK(-1) \ .
\end{gather}

\begin{thm}  \label{thm:min_K}  
One has \ $\min(\Delta_{\lg\gK\rg})= w_\gK(\Pi_\gK(-1))$. In other words, 
\[ 
\gamma\in\min(\Delta_{\lg\gK\rg})\Longleftrightarrow w_\gK^{-1}(\gamma)\in\Pi_\gK(-1) . 
\]
\end{thm}
\begin{proof}  We repeatedly use the following observation.
For $\gamma\in\Delta_{\lg\gK\rg}$, it follows from~\eqref{eq:delta} that 
if $w_\gK^{-1}(\gamma)>0$, then $w_\gK^{-1}(\gamma)\in \Delta^+_\gK(-1)$; whereas if
$w_\gK^{-1}(\gamma)<0$, then $w_\gK^{-1}(\gamma)\in -\Delta^+_\gK(2)$.

1$^o$. If $w_\gK^{-1}(\gamma)=\ap\in\Pi_\gK(-1)$, then $\gamma\in\Delta_{\lg\gK\rg}$. Assume further that 
$\gamma$ is not a minimal element of $\Delta_{\lg\gK\rg}$, i.e., $\gamma=\gamma'+\mu$  for some
$\gamma'\in\Delta_{\lg\gK\rg}$ and $\mu> 0$. Then $\ap=w_\gK^{-1}(\gamma')+w_\gK^{-1}(\mu)$ and 
there are two possibilities:
\begin{description}
\item[(a)] \   $w_\gK^{-1}(\gamma')<0$ and $w_\gK^{-1}(\mu)>0$;
\item[(b)] \  $w_\gK^{-1}(\gamma')>0$ and $w_\gK^{-1}(\mu)<0$.
\end{description}
For {\bf (a)}:  By~\eqref{eq:delta} and~\eqref{eq:bar-delta}, one has 
$w_\gK^{-1}(\gamma')\in \Delta^-_\gK(-2)$ and 
$w_\gK^{-1}(\mu)\in \Delta^+_\gK(0)\cup\Delta^+_\gK(-1)$. Then their sum belongs to 
$\Delta_\gK({\le}-2)=\Delta_{\gK}(-2)$, and this cannot be $\ap\in \Delta^+_\gK(-1)$. Hence this case is 
impossible.
\\
For {\bf (b)}:  By~\eqref{eq:delta} and~\eqref{eq:bar-delta}, one has 
$w_\gK^{-1}(\gamma')\in \Delta^+_\gK(-1)$ and 
$w_\gK^{-1}(\mu)\in \Delta^-_\gK(-2)\cup\Delta^-_\gK(-1)$. Then their sum again belongs to 
$\Delta_{\gK}(-2)$, which is impossible, too.

Thus, $\gamma=w_\gK(\ap)$ must be a minimal element of $\Delta_{\lg\gK\rg}$.

2$^o$.  Suppose that $\gamma\in\Delta_{\lg\gK\rg}$ and $w_\gK^{-1}(\gamma)\not\in \Pi_\gK(-1)$.
By~\eqref{eq:delta}, there is a dichotomy:
\begin{description}
\item[(a)] \   $w_\gK^{-1}(\gamma)>0$;
\item[(b)] \  $w_\gK^{-1}(\gamma)<0$.
\end{description}

For {\bf (a)}:  By~\eqref{eq:delta}, one has $w_\gK^{-1}(\gamma)\in\Delta^+_\gK(-1)$. Since 
$w_\gK^{-1}(\Delta_{\lg\gK\rg})\cap\Pi=\Pi_\gK(-1)$, we have $w_\gK^{-1}(\gamma)=\mu_1+\mu_2$ for
some $\mu_1,\mu_2>0$. From the $x_\gK$-grading of $\Delta^+$, we deduce that 
$\mu_1\in\Delta^+_\gK(-1)$ and $\mu_2\in\Delta^+_\gK(0)$. Letting $\gamma_i=w_\gK(\mu_i)$, we 
see that $\gamma_1\in\Delta_{\lg\gK\rg}$ and $\gamma_2>0$. Hence $\gamma=\gamma_1+\gamma_2$ 
is not a minimal element of $\Delta_{\lg\gK\rg}$.
\\ \indent
For {\bf (b)}: Here $w_\gK^{-1}(\gamma)\in -\Delta^+_\gK(2)$. Since $\theta\in\Delta^+_\gK(1)$, we have
$w_\gK^{-1}(\gamma)\ne -\theta$ and hence $w_\gK^{-1}(\gamma)=\nu_1-\nu_2$ for some 
$\nu_1,\nu_2>0$.  Using the $x_\gK$-grading of $\Delta^+$, one again encounters two possibilities:
\begin{itemize}
\item[\sf (i)] \ $\nu_1\in \Delta^+_\gK(-1), \nu_2\in \Delta^+_\gK(1)$;
\item[\sf (ii)] \ $\nu_1\in \Delta^+_\gK(0), \nu_2\in \Delta^+_\gK(2)$.
\end{itemize}
In both cases, $\gamma=w_\gK(\nu_1)+w_\gK(-\nu_2)$ and one of the summands lies in $\Delta_{\lg\gK\rg}$, 
while the other is positive.
Hence $\gamma\not\in\min(\Delta_{\lg\gK\rg})$.
\end{proof}

\begin{thm}       \label{thm:max_K}  
One has \ $\max(\ov{\Delta_{\lg\gK\rg}})= -w_\gK(\Pi_\gK(1))$. In other words,
\[ 
\gamma\in\max(\ov{\Delta_{\lg\gK\rg}})\Longleftrightarrow -w_\gK^{-1}(\gamma)\in\Pi_\gK(1) . 
\]
\end{thm}
\begin{proof}
To a great extent, the proof is analogous to that of Theorem~\ref{thm:min_K}, and we skip similar
arguments.  

1$^o$. If $-w_\gK^{-1}(\gamma)\in\Pi_\gK(1)$, then an argument similar to that in part 1$^o$ of
Theorem~\ref{thm:min_K} proves that $\gamma\in\max(\ov{\Delta_{\lg\gK\rg}})$.

2$^o$.  Conversely, suppose that $\gamma\in\ov{\Delta_{\lg\gK\rg}}$ and 
$-w_\gK^{-1}(\gamma)\not\in \Pi_\gK(1)$. In view of~\eqref{eq:bar-delta}, one has to handle two 
possibilities for $w_\gK^{-1}(\gamma)$.

{\bf (a)}: If $w_\gK^{-1}(\gamma)>0$, then $w_\gK^{-1}(\gamma)\in\Delta^+_\gK(0)$. Arguing as in the 
proof of Theorem~~\ref{thm:min_K} (part 2$^o${\bf (a)}), we show that $\gamma=\gamma_1-\gamma_2$, where 
$\gamma_1\in \ov{\Delta_{\lg\gK\rg}}$ and $\gamma_2>0$. Hence $\gamma\not\in\max(\ov{\Delta_{\lg\gK\rg}})$.

{\bf (b)}: If $w_\gK^{-1}(\gamma)<0$, then $-w_\gK^{-1}(\gamma)\in\Delta^+_\gK(1)\setminus\Pi_\gK(1)$.
Hence $-w_\gK^{-1}(\gamma)=\nu_1+\nu_2$ for some $\nu_1,\nu_2>0$. There again are two possibilities:
\begin{enumerate}
\item \  $\nu_1\in \Delta^+_\gK(0)$, \ \ $\nu_2\in \Delta^+_\gK(1)$; \ \quad 
[$(0,1)$-decomposition of \ $-w_\gK^{-1}(\gamma)$]
\item \  $\nu_1\in \Delta^+_\gK(-1)$, $\nu_2\in \Delta^+_\gK(2)$. \quad 
[$(-1,2)$-decomposition of \ $-w_\gK^{-1}(\gamma)$]
\end{enumerate}

In case (1), we get $\gamma=\gamma_2-\gamma_1$, where
$\gamma_2:=-w_\gK(\nu_2)\in \ov{\Delta_{\lg\gK\rg}}$ and $\gamma_1=w_\gK(\nu_1)>0$. Hence
$\gamma\not\in\max(\ov{\Delta_{\lg\gK\rg}})$.

In case (2), one similarly obtains the presentation of $\gamma$ as difference of two elements of 
$\Delta_{\lg\gK\rg}$, which is useless for us. However, one can replace such a $(-1,2)$-decomposition of 
$-w_\gK^{-1}(\gamma)$ with a $(0,1)$-decomposition, which is sufficient. Since 
$\nu_2\in \Delta^+_\gK(2)$, the root $\nu_2$ is long (Remark~\ref{rem:discussion}(3)) and not simple 
(for, $\Pi_\gK(2)=\varnothing$). Hence $(\nu_1,\nu_2)<0$ and $\nu_2=\nu_2'+\nu_2''$ with 
$\nu_1',\nu_2''>0$. W.l.o.g., we may assume that $(\nu_1,\nu'_2)<0$. One has three possibilities for the 
$x_\gK$-degrees of $(\nu_1',\nu_2'')$, i.e., $(1,1), (2,0), (0,2)$, and it is easily seen that 
$-w_\gK^{-1}(\gamma)=(\nu_1+\nu'_2)+\nu''_2$ is either a $(0,1)$-decomposition, or still a 
$(-1,2)$-decomposition, with $\nu''_2\in\Delta^+_\gK(2)$. But
in this last case we have $\htt(\nu''_2)< \htt(\nu_2)$, which provides the induction step.
\end{proof}

\begin{rmk}     \label{rem:verno-for-all-z}
Theorem~\ref{thm:min_K} holds for {\sl arbitrary} $z\in\De_{\sf ab}\cap\mathcal P^\vee$ 
in place of $x_\gK$, with certain amendments. That is, if $\theta(z)\le 1$, then the statement and the 
proof remain the same. If $\theta(z)=2$, then $-w_z(\theta)$ has to be added to $\min\Delta(\ah_z)$.
Certain complements of similar nature are also required in Theorem~\ref{thm:max_K}. I hope to elaborate on this topic in a subsequent publication.
\end{rmk}

\section{An explicit description of $w_\gK\in W$ and the ideals $\ah_\gK$}    
\label{sect:w_K&perebor}

\noindent
The element $w_\gK\in W$ and the abelian ideal $\ah_\gK$ have many interesting properties, which can 
be verified case-by-case. To this end, we need explicit formulae for $w_\gK$.  Our main tool is the 
following 
\begin{thm}      \label{thm:w^-1}
Given $z\in\De_{\sf ab}\cap\mathcal P^\vee$, suppose that $\rk\Delta_z(0)=\rk\Delta-1$ and $\theta(z)=1$. Then 
\[
   w_z^{-1}(\Pi)=\{\text{\normalfont the base of } \Delta_z(0) \ \text{ \normalfont in } \ \Delta^+_z(0)\}\cup \{-\theta\} . 
\]
\end{thm}
\begin{proof}
Set $p=\rk\Delta$, and let $\nu_1,\dots,\nu_{p-1}\in \Delta^+_z(0)$ be the base of $\Delta_z(0)$. 
By Proposition~\ref{lm:good-prop-w_K}, we have $w_z(\nu_i)\in \Pi$ ($i=1,\dots,p{-}1$) and 
$w_z(\theta)\in -\Pi$. The assertion follows.
\end{proof}
By Proposition~\ref{prop:fundam}, Theorem~\ref{thm:w^-1} applies to $z=x_\gK$. We demonstrate below
how to use this technique for finding $w_\gK\in GL(\te)$.

\begin{ex}    \label{ex:how-find-w_K}
(1) \ Let $\g$ be of type $\GR{D}{2n}$. Then $\Pi_\gK(0)=\varnothing$ and the base of $\Delta^+_\gK(0)$ 
corresponds to the edges of the Dynkin diagram, i.e., it consists of
$\ap_1+\ap_2, \ap_2+\ap_3,\dots, \ap_{2n-2}+\ap_{2n-1}, \ap_{2n-2}+\ap_{2n}$, cf. the diagram for 
$\GR{D}{2n}$ in Fig.~\ref{fig:class}. Therefore, the root system $\Delta_\gK(0)$ is of type 
$\GR{A}{n-1}+\GR{D}{n}$. More precisely, 

$\{\ap_1+\ap_2, \ap_3+\ap_4, \dots, \ap_{2n-3}+\ap_{2n-2}\}$ is a base for $\Delta(\GR{A}{n-1})$;

$\{\ap_2+\ap_3,\ap_4+\ap_5,\dots, \ap_{2n-4}+\ap_{2n-3}, \ap_{2n-2}+\ap_{2n-1},\ap_{2n-2}+\ap_{2n}\}$
is a base for $\Delta(\GR{D}{n})$.
\\
Since the Dynkin diagram of $\GR{A}{n-1}+\GR{D}{n}$ is obtained by removing the node $\ap_n$ from
the Dynkin diagram of $\GR{D}{2n}$, we must have $w_\gK^{-1}(\ap_n)=-\theta$. Then an easy argument shows that 
\begin{itemize}
\item $w_\gK^{-1}(\ap_1)=\ap_{2n-3}+\ap_{2n-2},\dots, w_\gK^{-1}(\ap_{n-2})=\ap_{3}+\ap_{4}, \  
w_\gK^{-1}(\ap_{n-1})=\ap_{1}+\ap_{2}$ \ -- \\for the $\GR{A}{n-1}$-part;
\item $w_\gK^{-1}(\ap_{n+1})=\ap_{2}+\ap_{3}, 
\dots, w_\gK^{-1}(\ap_{2n-2})=\ap_{2n-4}+\ap_{2n-3}, \\
w_\gK^{-1}(\{\ap_{2n-1},\ap_{2n}\})=\{\ap_{2n-2}+\ap_{2n-1}, \ \ap_{2n-2}+\ap_{2n-1}\}$ \ -- \ for the $\GR{D}{n}$-part;
\end{itemize}
The only unclear point for the $\GR{D}{n}$-part is how to distinguish $w_\gK^{-1}(\ap_{2n-1})$ and $w_\gK^{-1}(\ap_{2n})$. Using the expressions of simple roots of $\GR{D}{2n}$ via $\{\esi_i\}$, 
$i=1,\dots,2n$, we obtain two possibilities for $w_\gK$ as a signed permutation on $\{\esi_i\}$, where the only ambiguity concerns the sign of transformation $\esi_{2n}\mapsto \pm\esi_{2n}$. We then choose this
sign so that the total number of minuses be even, see Example~\ref{ex:w_K-so} below.
\\ \indent
(2) Similar argument works for the other orthogonal series.
\\ \indent
(3) For the exceptional Lie algebras, a certain ambiguity (due to the symmetry of the Dynkin diagram) occurs only for $\GR{E}{6}$.
\end{ex}

For the classical cases, our formulae for $w_\gK$ use the explicit standard models of $W$ as (signed) permutations on the set of $\{\esi_i\}$.  Write $\text{ord}(w_\gK(\g))$ for the order of $w_\gK=w_\gK(\g)$.
\begin{ex}         \label{ex:w_K-so}
Here we provide formulae for $w_\gK$ if $\g$ is an orthogonal Lie algebra.
\\
1$^o$. If $\g$ is of type $\GR{D}{2n}$, then

$w_\gK$: \  $\left( \text{\begin{tabular}{ccccc|ccccc}
$\esi_1$ & $\esi_3$ & $\dots$ & $\esi_{2n-3}$ & $\esi_{2n-1}$ & $\esi_2$ & $\esi_4$ & $\dots$ &
$\esi_{2n-2}$ & $\esi_{2n}$ \\
$\downarrow$ & $\downarrow$ & $\dots$ & $\downarrow$ & $\downarrow$ & 
$\downarrow$ & $\downarrow$ & $\dots$ & $\downarrow$ & $\downarrow$ \\
$-\esi_n$ & $-\esi_{n-1}$ & $\dots$ & $-\esi_{2}$ & $-\esi_{1}$ & $\esi_{n+1}$ & $\esi_{n+2}$ & $\dots$ 
& $\esi_{2n-1}$ & $(-1)^n\esi_{2n}$ \\
\end{tabular}}\right)$
\\[.8ex]
The last sign is determined by the condition that the total number of minuses must be even for 
type $\GR{D}{N}$.
\\[.6ex]
2$^o$. For $\g$ of type $\GR{B}{2n-1}$, one should merely omit the last column in the previous array.
\\[.6ex]
3$^o$. If $\g$ is of type $\GR{D}{2n+1}$, then the following adjustment works:

$w_\gK$: \  $\left( \text{\begin{tabular}{ccccc|cccccc}
$\esi_1$ & $\esi_3$ & $\dots$ & $\esi_{2n-3}$ & $\esi_{2n-1}$ & $\esi_2$ & $\esi_4$ & $\dots$ &
$\esi_{2n-2}$ & $\esi_{2n}$ & $\esi_{2n+1}$ \\
$\downarrow$ & $\downarrow$ & $\dots$ & $\downarrow$ & $\downarrow$ & 
$\downarrow$ & $\downarrow$ & $\dots$ & $\downarrow$ & $\downarrow$ & $\downarrow$ \\
$-\esi_n$ & $-\esi_{n-1}$ & $\dots$ & $-\esi_{2}$ & $-\esi_{1}$ & $\esi_{n+1}$ & $\esi_{n+2}$ & $\dots$ 
& $\esi_{2n-1}$ & $\esi_{2n}$ & $(-1)^n\esi_{2n+1}$\\
\end{tabular}}\right)$
\\[1ex]
4$^o$. For $\g$ of type $\GR{B}{2n}$, one should merely omit the last column in the previous array.

\textbullet \ \ 
It follows that in all four cases $w_\gK(\theta)=w_\gK(\esi_1+\esi_2)=-\esi_n+\esi_{n+1}=-\ap_n$, which 
agrees with Lemma~\ref{lm:good-prop-w_K}{\sf (ii)}.

\textbullet \ \ For $\GR{D}{2n+1}$, one has $\Pi_\gK(0)=\{\ap_{2n},\ap_{2n+1}\}$ (see Fig.~\ref{fig:class}) 
and $w_\gK$ takes $\Pi_\gK(0)$ to itself. Recall that here $\ap_{2n}=\esi_{2n}-\esi_{2n+1}$ and 
$\ap_{2n+1}=\esi_{2n}+\esi_{2n+1}$. The same happens for $\GR{B}{2n}$, where
$\Pi_\gK(0)=\{\ap_{2n}=\esi_{2n}\}$, cf. Lemma~\ref{lm:good-prop-w_K}{\sf (i)}.
\end{ex}

\begin{ex}         \label{ex:w_K-sl-sp}
Here we provide formulae for $w_\gK$ if $\g=\sltn$ or $\spn$.
\\
1$^o$. If $\g$ is of type $\GR{A}{2n-1}$, then
\\[.6ex]
\centerline{$w_\gK$: \  $\left( \text{\begin{tabular}{ccccc|ccccc}
$\esi_1$ & $\esi_2$ & $\dots$ & $\esi_{n-1}$ & $\esi_{n}$ & $\esi_{n+1}$ & $\esi_{n+2}$ & $\dots$ &
$\esi_{2n-1}$ & $\esi_{2n}$ \\
$\downarrow$ & $\downarrow$ & $\dots$ & $\downarrow$ & $\downarrow$ & 
$\downarrow$ & $\downarrow$ & $\dots$ & $\downarrow$ & $\downarrow$ \\
$\esi_{n+1}$ & $\esi_{n+2}$ & $\dots$ & $\esi_{2n-1}$ & $\esi_{2n}$ & $\esi_{1}$ & $\esi_2$ & $\dots$ 
& $\esi_{n-1}$ & $\esi_{n}$ \\
\end{tabular}}\right)$}

\noindent
It follows that $w_\gK(\theta)=w_\gK(\esi_1-\esi_{2n})=-\esi_n+\esi_{n+1}=-\ap_n$, which agrees with 
Lemma~\ref{lm:good-prop-w_K}{\sf (ii)}. Here $\Pi_\gK(0)=\Pi\setminus\{\ap_n\}$ (see Fig.~\ref{fig:class}) and 
$w_\gK(\ap_i)=\begin{cases}   \ap_{i+n},  & i<n \\ \ap_{i-n}, & i>n \end{cases}$ \ .
\\[1ex]
2$^o$. If $\g$ is of type $\GR{C}{n}$, then
\begin{center}
$w_\gK$: \  $\left( \text{\begin{tabular}{ccccc}
$\esi_1$ & $\esi_2$ & $\dots$ & $\esi_{n-1}$ & $\esi_{n}$  \\
$\downarrow$ & $\downarrow$ & $\dots$ & $\downarrow$ & $\downarrow$  \\
$-\esi_n$ & $-\esi_{n-1}$ & $\dots$ & $-\esi_{2}$ & $-\esi_{1}$ \\
\end{tabular}}\right)$
\end{center}
Here $w_\gK(\theta)=-\ap_n$, \ $\Pi_\gK(0)=\Pi\setminus\{\ap_n\}$, and $w_\gK(\ap_i)=\ap_{n-i}$ for $i<n$. 

\textbullet \ \  In both cases here, one has $w_\gK^2=1$.
\end{ex}
\begin{ex}    \label{ex:w_K-e6&f4}
1$^o$. For $\g$ of type $\GR{F}{4}$, we write $(a_1a_2a_3a_4)$ for $\sum_{i=1}^4 a_i\ap_i$. Then
$w_\gK(\ap_i)=\ap_i$ for $i=1,2$ and $w_\gK(\ap_3)=-(2421)$, $w_\gK(\ap_4)=(2431)=\theta-\ap_4$.
It follows that $w_\gK(\theta)=-\ap_4$.

2$^o$. For $\g$ of type $\GR{E}{6}$, we write $(a_1a_2a_3a_4a_5a_6)$ for
$\gamma=\sum_{i=1}^6 a_i\ap_i=$
\raisebox{-2.1ex}{\begin{tikzpicture}[scale= .85, transform shape]
\node (b) at (0,0) {$a_1$}; \node (c) at (.4,0) {$a_2$};
\node (d) at (.8,0) {$a_3$}; \node (e) at (1.2,0) {$a_4$}; \node (f) at (1.6,0) {$a_5$};
\node (g) at (.8,-.4) {$a_6$};
\end{tikzpicture}}. 
Then
$w_\gK(\ap_i)=\ap_i$ for $i=1,2,4,5$ and $w_\gK(\ap_3)=-(122211)$, $w_\gK(\ap_6)=(123211)=\theta-\ap_6$. It follows that $w_\gK(\theta)=-\ap_6$.

3$^o$. For $\g$ of type $\GR{G}{2}$, we have $\Pi_l=\{\ap_2\}$, $w_\gK=(r_{\ap_2}r_{\ap_1})^2$, and
$w_\gK(\theta)=-\ap_2$.

\textbullet \ \  In all three cases, one has $w_\gK^3=1$.
\end{ex}
\begin{ex}    \label{ex:w_K-e7}
1$^o$. For $\GR{E}{7}$, we have $\Pi_\gK(-1)=\{\ap_2,\ap_4,\ap_6\}$ (see Fig.~\ref{fig:except}) and 

{\tabcolsep=0.15em
\begin{tabular}{c|c c c c c c c|}    
$\ap_i$ & $\ap_1$ & $\ap_2$ & $\ap_3$ & $\ap_4$ & $\ap_5$ & $\ap_6$ & $\ap_7$ \\ \hline
$w_\gK(\ap_i)$ &  \rule{0pt}{3ex} 
\raisebox{-2.1ex}{\begin{tikzpicture}[scale= .85, transform shape]
\node (a) at (-.1,0) {--1}; \node (b) at (.2,0) {2}; \node (c) at (.4,0) {2};
\node (d) at (.6,0) {2}; \node (e) at (.8,0) {1}; \node (f) at (1,0) {0};
\node (g) at (.6,-.4) {1};
\end{tikzpicture}} &
\raisebox{-2.1ex}{\begin{tikzpicture}[scale= .85, transform shape]
\node (a) at (0,0) {1}; \node (b) at (.2,0) {2}; \node (c) at (.4,0) {2};
\node (d) at (.6,0) {2}; \node (e) at (.8,0) {1}; \node (f) at (1,0) {1};
\node  at (.6,-.4) {1};
\end{tikzpicture}} & 
\raisebox{-2.1ex}{\begin{tikzpicture}[scale= .85, transform shape]
\node (a) at (-.1,0) {--1}; \node (b) at (.2,0) {1}; \node (c) at (.4,0) {2};
\node (d) at (.6,0) {2}; \node (e) at (.8,0) {1}; \node (f) at (1,0) {1};
\node (g) at (.6,-.4) {1};
\end{tikzpicture}} &
\raisebox{-2.1ex}{\begin{tikzpicture}[scale= .85, transform shape]
\node (a) at (0,0) {1}; \node (b) at (.2,0) {1}; \node (c) at (.4,0) {2};
\node (d) at (.6,0) {2}; \node (e) at (.8,0) {2}; \node (f) at (1,0) {1};
\node  at (.6,-.4) {1};
\end{tikzpicture}} &
\raisebox{-2.1ex}{\begin{tikzpicture}[scale= .85, transform shape]
\node (a) at (-.1,0) {--0}; \node (b) at (.2,0) {1}; \node (c) at (.4,0) {2};
\node (d) at (.6,0) {2}; \node (e) at (.8,0) {2}; \node (f) at (1,0) {1};
\node (g) at (.6,-.4) {1};
\end{tikzpicture}} &
\raisebox{-2.1ex}{\begin{tikzpicture}[scale= .85, transform shape]
\node (a) at (0,0) {0}; \node (b) at (.2,0) {1}; \node (c) at (.4,0) {2};
\node (d) at (.6,0) {3}; \node (e) at (.8,0) {2}; \node (f) at (1,0) {1};
\node  at (.6,-.4) {1};
\end{tikzpicture}} &
\raisebox{-2.1ex}{\begin{tikzpicture}[scale= .85, transform shape]
\node (a) at (-.1,0) {--1}; \node (b) at (.2,0) {1}; \node (c) at (.4,0) {1};
\node (d) at (.6,0) {2}; \node (e) at (.8,0) {2}; \node (f) at (1,0) {1};
\node (g) at (.6,-.4) {1};
\end{tikzpicture}}   \\
\end{tabular} }

\noindent
It follows that $w_\gK(\theta)=-\ap_7$. A direct calculation shows that $\text{ord}(w_\gK)=18$.

\vspace{1ex}
2$^o$. For $\GR{E}{8}$, we have $\Pi_\gK(-1)=\{\ap_1,\ap_3,\ap_5,\ap_7\}$ (see Fig.~\ref{fig:except}) 
and 

{\tabcolsep=0.15em
\begin{tabular}{c|c c c c c c c c|}     
$\ap_i$ & $\ap_1$ & $\ap_2$ & $\ap_3$ & $\ap_4$ & $\ap_5$ & $\ap_6$ & $\ap_7$ & $\ap_8$ \\ \hline
$w_\gK(\ap_i)$ &  \rule{0pt}{3ex} 
\raisebox{-2.1ex}{\begin{tikzpicture}[scale= .85, transform shape]
\node (a) at (0,0) {0}; \node (b) at (.2,0) {1}; \node (c) at (.4,0) {2};
\node (d) at (.6,0) {3}; \node (e) at (.8,0) {4}; \node (f) at (1,0) {3}; \node (g) at (1.2,0) {1};
\node (h) at (.8,-.4) {2};
\end{tikzpicture}}  & 
\raisebox{-2.1ex}{\begin{tikzpicture}[scale= .85, transform shape]
\node (a) at (-.1,0) {--0}; \node (b) at (.2,0) {1}; \node (c) at (.4,0) {2};
\node (d) at (.6,0) {3}; \node (e) at (.8,0) {4}; \node (f) at (1,0) {2}; \node (g) at (1.2,0) {1};
\node (h) at (.8,-.4) {2};
\end{tikzpicture}}  & 
\raisebox{-2.1ex}{\begin{tikzpicture}[scale= .85, transform shape]
\node (a) at (0,0) {1}; \node (b) at (.2,0) {1}; \node (c) at (.4,0) {2};
\node (d) at (.6,0) {3}; \node (e) at (.8,0) {4}; \node (f) at (1,0) {2}; \node (g) at (1.2,0) {1};
\node (h) at (.8,-.4) {2};
\end{tikzpicture}}  & 
\raisebox{-2.1ex}{\begin{tikzpicture}[scale= .85, transform shape]
\node (a) at (-.1,0) {--1}; \node (b) at (.2,0) {1}; \node (c) at (.4,0) {2};
\node (d) at (.6,0) {3}; \node (e) at (.8,0) {3}; \node (f) at (1,0) {2}; \node (g) at (1.2,0) {1};
\node (h) at (.8,-.4) {2};
\end{tikzpicture}}  & 
\raisebox{-2.1ex}{\begin{tikzpicture}[scale= .85, transform shape]
\node (a) at (0,0) {1}; \node (b) at (.2,0) {2}; \node (c) at (.4,0) {2};
\node (d) at (.6,0) {3}; \node (e) at (.8,0) {3}; \node (f) at (1,0) {2}; \node (g) at (1.2,0) {1};
\node (h) at (.8,-.4) {2};
\end{tikzpicture}}  & 
\raisebox{-2.1ex}{\begin{tikzpicture}[scale= .85, transform shape]
\node (a) at (-.1,0) {--1}; \node (b) at (.2,0) {2}; \node (c) at (.4,0) {2};
\node (d) at (.6,0) {3}; \node (e) at (.8,0) {3}; \node (f) at (1,0) {2}; \node (g) at (1.2,0) {1};
\node (h) at (.8,-.4) {1};
\end{tikzpicture}}  & 
\raisebox{-2.1ex}{\begin{tikzpicture}[scale= .85, transform shape]
\node (a) at (0,0) {1}; \node (b) at (.2,0) {2}; \node (c) at (.4,0) {3};
\node (d) at (.6,0) {3}; \node (e) at (.8,0) {3}; \node (f) at (1,0) {2}; \node (g) at (1.2,0) {1};
\node (h) at (.8,-.4) {1};
\end{tikzpicture}}  & 
\raisebox{-2.1ex}{\begin{tikzpicture}[scale= .85, transform shape]
\node (a) at (-.1,0) {--1}; \node (b) at (.2,0) {2}; \node (c) at (.4,0) {2};
\node (d) at (.6,0) {2}; \node (e) at (.8,0) {3}; \node (f) at (1,0) {2}; \node (g) at (1.2,0) {1};
\node (h) at (.8,-.4) {2};
\end{tikzpicture}} 
\\
\end{tabular} }

\noindent
It follows that $w_\gK(\theta)=-\ap_7$. A direct calculation shows that $\text{ord}(w_\gK)=5$.
\end{ex}

\begin{rmk}
(1) We do not know a general formula for $\text{ord}(w_\gK(\soN))$. 
Using 
Example~\ref{ex:w_K-so}, it is not hard to prove that $\text{ord}(w_\gK)$ takes the same value for
$\GR{B}{2n-1},\GR{D}{2n},\GR{B}{2n},\GR{D}{2n+1}$. But explicit computations, up to $n=13$, show that
the function $n\mapsto \text{ord}(w_\gK(\GR{D}{2n}))$ behaves rather chaotically. 

(2) 
If $\ap\in\Pi_\gK(0)$, then $w_\gK(\ap)\in\Pi_\gK(0)$ as well. But this does not hold for arbitrary $z\in\De_{\sf ab}$. In general, it can happen that $\ap\in\Pi_z(0)$, but
$w_z(\ap)\in\Pi\setminus\Pi_z(0)$.
\end{rmk}
The main result of this section is an explicit uniform description of $\Delta(\ah_\gK)$. 

\begin{thm}       \label{thm:ab-ideal}
Set \ $d_\gK=\htt(\theta)+1-\sum_{\ap\in\Pi_\gK(-1)} [\theta:\ap]=
1+\sum_{\ap\in\Pi_\gK({\ge}0)} [\theta:\ap]$. Then
\begin{itemize}
\item[\sf (i)] \ $\htt(\gamma)=d_\gK$ if and only if\/ $w_\gK^{-1}(\gamma)\in\Pi_\gK(-1)$;
\item[\sf (ii)] \ $\htt(\gamma)=d_\gK-1$ if and only if \  $-w_\gK^{-1}(\gamma)\in\Pi_\gK(1)$;
\item[\sf (iii)] \ $\displaystyle \Delta(\ah_\gK)=\{\gamma\in\Delta^+\mid\htt(\gamma)\ge d_\gK\}$.
\end{itemize}
\end{thm}
\begin{proof}
{\bf 1}$^o$. For $\GR{A}{2n-1}$ and $\GR{C}{n}$, $x_\gK$ is dominant. Hence 
$\Pi_\gK(-1)=\varnothing$, $d_\gK=\htt(\theta)+1=\mathsf{h}$ is the Coxeter number, and 
$\Delta(\ah_\gK)=\varnothing$, which agrees with {\sf (iii)}.  Here $\theta$ is the only root of height 
$d_\gK-1$, $\Pi_\gK(1)=\{\ap_n\}$, and $w_\gK(\ap_n)=-\theta$, see Example~\ref{ex:w_K-sl-sp}. This 
confirms {\sf (ii)}, whereas {\sf (i)} is vacuous.

{\bf 2}$^o$.  In the other cases, $\Pi_\gK(-1)\ne\varnothing$.

{\sf (i)} \ Using formulae from Examples~\ref{ex:w_K-so}--\ref{ex:w_K-e7}, one directly verifies that
 if $\ap\in\Pi_\gK(-1)$, then $\htt( w_\gK(\ap))=d_\gK$. It also happens that
 $\#\{\gamma\mid \htt(\gamma)=d_\gK\}=\#\Pi_\gK(-1)$ in all cases.

{\sf (ii)} \ Again, this can be verified directly. Alternatively, this follows from {\sf (i)}, Theorem~\ref{thm:min_K}, and Theorem~\ref{thm:max_K} (without further verifications).

{\sf (iii)} \ This follows from {\sf (i)} and Theorem~\ref{thm:min_K}.
\end{proof}

Note that \
$\displaystyle \mathsf h-1=\htt(\theta)= \sum_{\ap\in\Pi_\gK(1)}[\theta:\ap]+\sum_{\ap\in\Pi_\gK(0)}[\theta:\ap]+
\sum_{\ap\in\Pi_\gK(-1)}[\theta:\ap]$ 
\[
\text{ and } \quad 1=\theta(x_\gK)= \sum_{\ap\in\Pi_\gK(1)}[\theta:\ap]-\sum_{\ap\in\Pi_\gK(-1)}[\theta:\ap] \ .
\] 
Therefore, if $\Pi_\gK(0)=\varnothing$, then $d_\gK=(\mathsf h/2)+1$; whereas,
if $\Pi_\gK(0)\ne\varnothing$, then $d_\gK> (\mathsf h/2)+1$. 

\begin{rmk}
A remarkable property of the abelian ideal $\ah_\gK$ is that 
$\min\Delta(\ah_\gK)$ consists of {\bf all} roots of a {\bf fixed} height. This can be explained by the properties that
$\Pi_\gK(-1)\cup \Pi_\gK(1)$ is connected in the Dynkin diagram and that $+1$ and $-1$ nodes alternate.
For, if $\ap_i$ and $\ap_j$ are adjacent nodes, $\ap_i\in\Pi_\gK(-1)$, and $\ap_j\in\Pi_\gK(1)$, then 
$\ap_i+\ap_j$ is a simple root in $\Delta^+_\gK(0)$. Hence $w_\gK(\ap_i)\in \min\Delta(\ah_\gK)$,
$-w_\gK(\ap_j)\in \max(\Delta^+\setminus \Delta(\ah_\gK))$, and
$w_\gK(\ap_i+\ap_j)\in\Pi$. Hence
$\htt(w_\gK(\ap_i))=1+\htt(-w_\gK(\ap_j))$. In view of Theorems~\ref{thm:min_K} and \ref{thm:max_K},
together with the connectedness and the alternating property for $\Pi_\gK(-1)\cup \Pi_\gK(1)$, this 
relation propagates to any pair in $\min\Delta(\ah_\gK)\times \max(\Delta^+\setminus \Delta(\ah_\gK))$.
\\ \indent
But the exact value of the boundary height, which is $d_\gK$ is our case, has no explanation.
\end{rmk}

\section{The involution of $\g$ associated with the cascade}
\label{sect7:involution}

In this section, we assume that $\spx\in\BZ$, i.e., $\g\ne \mathfrak{sl}_{2n+1}$. Then 
partition~\eqref{eq:main-partition} yields the partition of the whole root system
$\Delta=\bigcup_{i=-2}^2 \Delta_{\gK}(i)$ such that $\Delta_{\gK}(2)=\Delta^+_\gK(2)$. 

Set  $\Delta_0=\Delta_{\gK}(-2)\cup \Delta_{\gK}(0)\cup \Delta_{\gK}(2)$ and 
$\Delta_1=\Delta_{\gK}(-1)\cup \Delta_{\gK}(1)$. Consider the vector space decomposition
$\g=\g_0\oplus\g_1$, where $\g_0=\te\oplus (\oplus_{\gamma\in\Delta_0}\g_\gamma)$ and 
$\g_1=\oplus_{\gamma\in\Delta_1}\g_\gamma$. This is a $\BZ_2$-grading and the corresponding
involution of $\g$, denoted $\sigma_\gK$, is inner.

\begin{lm}   \label{lm:involution}
The involution $\sigma_\gK$ has the property that
 \\ \centerline{
$\dim\g_0=\dim\be-\#\gK$ \ and \ $\dim\g_1=\dim\ut+\#\gK$.
} 
\end{lm}
\begin{proof}
This readily follows from the symmetry of $\spx$ on the Frobenius envelope $\be_\gK$ and the equality 
$\dim\be_\gK=\dim\ut+\#\gK$, see Section~\ref{sect:frob} or Lemma~\ref{lm:symmetry-value}. For,  one 
has
\\[.6ex]
\centerline{
\begin{tabular}{c||c|c|c|c|c}
$i$ & $-2$ & $-1$ & 0 & 1 & 2  \\ \hline
$\#\Delta_{\gK}^+(i)$ & 0 & $a$ & $b{-}\#\gK$ & $b$ & $a$ \\
$\#\Delta_{\gK}^-(i)$ & $a$ & $b$ & $b{-}\#\gK$ & $a$ & $0$ \\
\end{tabular}} 
\\[.8ex]  
for some $a, b$. Then $\dim\ut=2a+2b-\#\gK$,  
$\dim\g_0=2a+2b+\rk\g-2\#\gK$ and $\dim\g_1=2a+2b$.
\end{proof}

Since $\ind\ut=\#\gK$ and $\ind\be+\ind\ut=\rk\g$, the formulae of Lemma~\ref{lm:involution} mean 
that $\dim\g_1=\dim\ut+\ind\ut=\dim\be-\ind\be$ and 
$\dim\g_0=\dim\ut+\ind\be=\dim\be-\ind\ut$. 

Recall that $\te_\gK=\bigoplus_{i=1}^m [e_{\beta_i}, e_{-\beta_i}]\subset \te$.

\begin{lm}    \label{lm:t_K-regul-ss}
The subalgebra $\te_\gK$ contains a regular semisimple element of\/ $\g$. 
\end{lm}
\begin{proof}  
By Eq.~\eqref{eq:decomp-Delta}, if $\gamma\in\Delta^+$, then $\gamma\in \gH_{\beta_j}$ 
for a unique $\beta_j\in\gK$. That is, for any $\gamma\in\Delta^+$, there exists $\beta_j\in\gK$ such that
$(\gamma,\beta_j)>0$.  Therefore, there is a $\nu\in\lg \beta_1,\dots,\beta_m\rg_\BQ$ such that 
$(\gamma,\nu)\ne 0$ for {\bf every} $\gamma\in\Delta^+$. Upon the identification of $\te$ and $\te^*$, 
this yields a required element of $\te_\gK$.
\end{proof}
\begin{thm}    \label{thm:sigma-S&N}
The involution $\sigma_\gK$ has the property that $\g_1$ contains a regular semisimple and a regular
nilpotent element of $\g$. 
\end{thm}
\begin{proof}
For each $\beta_i\in\gK$, consider the subalgebra $\tri(\beta_i)$ with basis
$\{e_{\beta_i}, e_{-\beta_i}, [e_{\beta_i}, e_{-\beta_i}]\}$. Then $\tri(\beta_i)\simeq \tri$ and since the 
elements of $\gK$ are strongly orthogonal, all these  $\tri$-subalgebras pairwise commute. Hence 
$\h=\bigoplus_{j=1}^m \tri(\beta_j)$ is a Lie algebra and $\te_\gK$ is a Cartan subalgebra of $\h$. 
Recall that $\beta_j(x_\gK)=1$, i.e., $\beta_j\in\Delta_{\gK}^+(1)$ for each $j$. By the 
very definition of $\sigma_\gK$, this means that
$\h\cap\g_0=\te_\gK$ and $\h\cap\g_1=(\bigoplus_{i=1}^m \g_{\beta_i})\oplus
(\bigoplus_{i=1}^m \g_{-\beta_i})$. Clearly, there is a Cartan subalgebra of $\h$ that is contained in
$\h\cap\g_1$
(because this is true for each $\tri(\beta_i)$ separately). Combining this with 
Lemma~\ref{lm:t_K-regul-ss}, we see that $\g_1$ contains a regular semisimple element of $\g$.

Finally, for any involution $\sigma$, its $(-1)$-eigenspace $\g_1$ contains a regular semisimple element if and only if it contains a regular nilpotent element.
\end{proof}

\begin{rmk}   \label{rem:sigma-N-reg}
{\sf (i)} \ One can prove that if $\ce$ is a Cartan subalgebra of $\h=\bigoplus_{j=1}^m \tri(\beta_j)$ that is contained in $\h\cap\g_1$, then
$\ce$ is a maximal diagonalisable subalgebra of $\g_1$. In other words, $\ce\subset\g_1$ is a 
{\it Cartan subspace\/} associated with $\sigma_\gK$. Therefore, the rank of the symmetric variety
$G/G_0$ equals $\dim\ce=\#\gK$.
\\ \indent
{\sf (ii)} \ The involution $\sigma_\gK$ is the unique, up to $G$-conjugacy, {\bf inner} involution such that 
$\g_1$ contains a regular nilpotent element. Moreover, $\sigma_\gK$ has the property that
$\co\cap\g_1\ne\varnothing$ for {\bf any nilpotent} $G$-orbit $\co\subset\g$, see~\cite[Theorem\,3]{an}.
\\ \indent
{\sf (iii)} \ If $\#\gK=\rk\g$, then $\dim\g_1=\dim\be$, $\dim\g_0=\dim\ut$, and $\g_1$ contains 
a Cartan subalgebra of $\g$. In this case, $\sigma_\gK$ is an {\it involution of maximal rank} and 
one has a stronger assertion that $\co\cap\g_1\ne\varnothing$ for {\bf any} $G$-orbit $\co$ in $\g$
(\cite[Theorem\,2]{an}).
\end{rmk}

\section{The nilpotent $G$-orbit associated with the cascade} 
\label{sect:nilp-orb}

\noindent
Consider the nilpotent element associated with $\gK$ 
\beq    \label{eq:e_K}
      e_{\gK}:=\sum_{\beta\in\gK} e_\beta=\sum_{i=1}^m e_{\beta_i}\in\ut^+
\eeq
Since the roots in $\gK$ are linearly independent, the closure of $G{\cdot}e_\gK$ contains the space 
$\bigoplus_{\beta\in\gK}\g_\beta$. Hence the nilpotent orbit $G{\cdot}e_\gK$ does not depend on the 
choice of root vectors $e_\beta\in\g_\beta$. The orbit $\co_\gK:=G{\cdot}e_\gK$ is said to be the 
{\it cascade (nilpotent) orbit}. Our goal is to obtain some properties of this orbit. 

Recall from Section~\ref{sect:comb-prop} that if $\theta$ is fundamental, then $\tap$ is the only (long!) 
simple root such that $(\theta,\tap)>0$ and $\tilde\gK=\{\beta\in\gK\mid (\beta,\tap)<0\}$.
It then follows from~\eqref{eq:tap} that $\#\tilde\gK\le 3$ and
$\# \tilde\gK=3$ if and only if the roots in $\tilde{\gK}$ are long.
Actually, one has $\#\tilde\gK=1$ for $\GR{G}{2}$, $\#\tilde\gK=2$ for $\GR{B}{3}$, and $\#\tilde\gK=3$ 
for the remaining cases with fundamental $\theta$.

\begin{thm}        \label{prop:ad^4}
1$^o$. Suppose that $\theta$ is fundamental, and let $\tilde\ap$ be the unique simple root such that 
$(\theta,\tilde\ap)\ne 0$. Then  {\sf (i)} \ 
$(\ad e_\gK)^4(e_{-\theta+\tilde\ap})\ne 0$ and {\sf (ii)} \ 
$(\ad e_\gK)^5= 0$.

2$^o$. If $\theta$ is {\bf not} fundamental, then $(\ad e_\gK)^3= 0$.
\end{thm}
\begin{proof}
1$^o${\sf (i)}. \  It follows from Eq.~\eqref{eq:e_K} that
\[
   (\ad e_\gK)^4=\sum_{i_1,i_2,i_3,i_4}\ad\!(e_{\beta_{i_1}}){\cdot}\ldots \cdot \ad\!(e_{\beta_{i_4}}),
\]
where the sum is taken over all possible quadruples of indices from $\{1,\dots,m\}$. Set 
\[
   \ca_{i_1,i_2,i_3,i_4}=\ad\!(e_{\beta_{i_1}})\ad\!(e_{\beta_{i_2}})\ad\!(e_{\beta_{i_3}})\ad\!(e_{\beta_{i_4}}) .
\]
As the roots in $\gK$ are strongly orthogonal, the ordering of factors in $\ca_{i_1,i_2,i_3,i_4}$ is 
irrelevant. Hence the operator $\ca_{i_1,i_2,i_3,i_4}$ depends only on the $4$-multiset 
$\{i_1,i_2,i_3,i_4\}$. Furthermore, the {\bf nonzero} operators corresponding to different $4$-multisets 
are linearly independent. Therefore, to ensure that $(\ad e_\gK)^4\ne 0$, it suffices to point out a 
$4$-multiset $\eus M$ and a root vector $e_\gamma$ such that $\ca_\eus M(e_\gamma)\ne 0$. Of 
course, in place of $4$-multisets of indices in $[1,m]$, one can deal with $4$-multisets in $\gK$.

Using~\eqref{eq:tap} and $\tilde\gK\subset\gK$, one defines a natural $4$-multiset $\tilde{\eus M}$ in 
$\gK$. The first element of $\tilde{\eus M}$ is $\theta=\beta_1$ and then one takes each 
$\beta_i\in\tilde{\gK}$ with multiplicity $k_i=(\tap,\tap)/(\beta_i,\beta_i)$. The resulting $4$-multiset has 
the property that $(-\theta+\tap)+(\theta + \sum_{i\in J}k_i\beta_i)=\theta-\tap$ and 
$(-\theta+\tap,\beta)<0$ for any $\beta$ in $\tilde{\eus M}$. This implies that
\[
       0\ne  \ca_{\tilde{\eus M}}(e_{-\theta+\tap})\in \g_{\theta-\tap} .
\] 
\indent
{\sf (ii)} \ Now we deal with $5$-multisets of $\gK$. Assume that
$\tilde{\eus M}=\{\beta_{i_1},\dots,\beta_{i_5}\}$ and  $\ca_{\tilde{\eus M}}\ne 0$.
Then there are $\gamma,\mu\in \Delta$ such that
$0\ne \ca_{\tilde{\eus M}}(\g_{-\mu})\subset \g_\gamma$. Hence
$\gamma+\mu=\sum_{j=1}^5 \beta_{i_j}$ and $(\gamma+\mu)(x_\gK)=5$. But $\gamma(x_\gK)\le 2$
for any $\gamma\in\Delta$ (Theorem~\ref{thm:spektr-fonin}(3)). This contradiction shows that 
$(\ad e_\gK)^5=0$.

2$^o$. By Theorem~\ref{thm:spektr-fonin}(1)(2), we have $\gamma(x_\gK)\le 1$ for any $\gamma\in\Delta$. Hence the equality
$\gamma+\mu=\sum_{j=1}^3 \beta_{i_j}$ is impossible. This implies that $\ca_{\tilde{\eus M}}= 0$
for any $3$-miltiset $\tilde{\eus M}$ of $\gK$ and thereby $(\ad e_\gK)^3=0$.
\end{proof}

Let $\N\subset\g$ be the set of nilpotent elements. Recall that the {\it height\/} of $e\in\N$, denoted 
$\hot(e)$ or $\hot(G{\cdot}e)$, is the maximal $l\in\BN$ such that $(\ad e)^l\ne 0$ 
(see~\cite[Section\,2]{p99}). By the Jacobson--Morozov theorem, $(\ad e)^2\ne 0$ for any $e\in\N$, i.e., 
$\hot(G{\cdot}e)\ge 2$.
\\ \indent
The {\it complexity\/} of a $G$-variety $X$, $c_G(X)$, is the minimal codimension of the $B$-orbits in 
$X$. If $X$ is irreducible, then $c_G(X)=\trdeg\BC(X)^B$, where $\BC(X)^G$ is the field of 
$B$-invariant rational functions on $X$. The {\it rank\/} of an irreducible $G$-variety $X$ is defined by the 
equality $c_G(X)+r_G(X)=\trdeg\BC(X)^U$~\cite{these}. If $c_G(X)=0$, then $X$ is said to be  {\it spherical}. 

\begin{prop}    \label{thm:sferic-orbit}
\leavevmode\par
\begin{itemize}
\item[{\sf (i)}]  If $\theta$ is fundamental, then the orbit $G{\cdot}e_\gK$ is {\bf not} spherical and\/ 
$\hot(G{\cdot}e_\gK)=4$. 
\item[{\sf (ii)}]  If $\theta$ is {\bf not} fundamental, then the orbit
$G{\cdot}e_\gK$ is spherical and\/ $\hot(G{\cdot}e_\gK)=2$.
\end{itemize}
\end{prop}
\begin{proof}
By~\cite[Theorem\,(0.3)]{p94}, a nilpotent orbit $G{\cdot}e$ is spherical if and only if $(\ad e)^4=0$. 
Hence both assertions follow from Theorem~\ref{prop:ad^4}. 
\end{proof}

\subsection{A description of the cascade orbits} 
\label{subs:cascade-orbits}
For the classical Lie algebras, we determine the partition corresponding to $\co_{\gK}$. While for the 
exceptional Lie algebras, we point out a "minimal including regular subalgebra" in the sense 
of Dynkin~\cite{dy}.

{\bf I.} In the classical cases, we use formulae for the height of
nilpotent elements of $\slv$ or $\sov$ or $\spv$ in terms of the corresponding partitions of $\dim\BV$, 
see~\cite[Theorem\,2.3]{p99}.

\textbullet\quad $\g=\mathfrak{so}_{2N+1}$. Since $\hot(\co_\gK)=4$, the parts of $\blb(e_\gK)$ does 
not exceed $3$, i.e., $\blb(e_\gK)=(3^a, 2^b,1^c)$ with $a>0$ and $3a+2b+c=2N+1$. Then
$\blb(e_\gK^2)=(2^a,1,\dots,1)$. Hence
$\rk(e_\gK)=2a+b$ and $\rk(e_\gK^2)=a$. On the other hand, here $\#\gK=N=\rk\g$ and
using the formulae for roots in $\gK$ and thereby the explicit matrix form for $e_\gK$, one readily computes that
\\
\centerline{
$\rk(e_\gK)=\begin{cases}   N, & \text{ if $N$ is even} \\
N{+}1, & \text{ if $N$ is odd}   \end{cases}$ \quad and \ 
$\rk(e_\gK^2)=\frac{1}{2}\rk(e_\gK)=  [(N{+}1)/2]$. 
}
Therefore, if $N$ is either $2j{-}1$ or $2j$, then $a=j$ and $b=0$. 

Hence $\blb(e_\gK)=(3^j, 1^{j-1})$ or $(3^{j}, 1^{j+1})$, respectively.

\textbullet\quad $\g=\mathfrak{so}_{2N}$. Here  $\hot(\co_\gK)=4$ and $\blb(e_\gK)=(3^a, 2^b,1^c)$, 
with $a>0$ and $3a+2b+c=2N$. Then again $\rk(e_\gK)=2a+b$ and $\rk(e_\gK^2)=a$. The explicit form
of $\gK$ and $e_\gK$ shows that
\\
\centerline{
$\rk(e_\gK)=\begin{cases}   N, & \text{ if $N$ is even} \\
N{-}1, & \text{ if $N$ is odd}   \end{cases}$ \quad and \ 
$\rk(e_\gK^2)=\frac{1}{2}\rk(e_\gK)=  [N/2]$. }
\\
Therefore, $a=[N/2]$ and $b=0$ in both cases. Hence
\\[.6ex]
\centerline{
$\blb(e_\gK)=(3^j, 1^j)$, if $N=2j$; \ $\blb(e_\gK)=(3^j, 1^{j+2})$, if $N=2j+1$.}

\textbullet\quad $\g=\mathfrak{sl}_{N+1}$ or $\mathfrak{sp}_{2N}$. Then $\hot(\co_\gK)=2$, 
$\blb(e_\gK)=(2^a,1^b)$, and $a=\rk(e_\gK)=\#\gK$. Therefore, $\blb(e_\gK)=(2^n)$ for $\sltn$ and 
$\spn$, while $\blb(e_\gK)=(2^n,1)$ for $\mathfrak{sl}_{2n+1}$.

{\bf II.} In the exceptional cases, one can use some old, but extremely helpful computations 
of {E.B.\,Dynkin.} Following Dynkin, we say that a subalgebra $\h$ of $\g$ is {\it regular}, if it is 
normalised by a Cartan subalgebra.  As in Section~\ref{sect7:involution}, consider 
$\h:=\bigoplus_{i=1}^m \tri(\beta_i)$. Then $\h$ is normalised by $\te$
and $e_\gK\in \h$ is a regular nilpotent element of $\h$. Clearly, $\h$ is a minimal regular semisimple 
subalgebra of $\g$ meeting $\co_\gK$.

For every nilpotent $G$-orbit $\co$ in an exceptional Lie algebra $\g$, Dynkin computes all, up to 
conjugacy, minimal regular semisimple subalgebras of $\g$ meeting $\co$, see Tables~16--20 in \cite{dy}. 
(This information is also reproduced, with a few corrections, in Tables 2--6 in~\cite{age75}.) Therefore, it 
remains only to pick the nilpotent orbit with a "minimal including regular subalgebra" of the required type, 
which also yields the corresponding {\it weighted Dynkin diagram\/} $\gD(\co_\gK)$.

For the regular subalgebras $\tri\subset\g$, Dynkin uses the Cartan label $\GR{A}{1}$
(resp. $\GRt{A}{1}$) if the corresponding root is long (resp. short). Therefore, we are looking for the 
"minimal including regular subalgebra" of type $m\GR{A}{1}$, $m=\#\gK$, if $\gK\subset \Delta_l$; whereas for $\g$ of type 
$\GR{G}{2}$ we need the  subalgebra of type $\GR{A}{1}+\GRt{A}{1}$.

\subsection{Another approach to $\co_\gK$}     
\label{subs:charact-e_k}
By the very definition of $x_\gK\in\te$ and $e_\gK$, we have $[x_\gK,e_\gK]=e_\gK$. It is also clear that 
$x_\gK\in \Ima(\ad e_\gK)$. Therefore $h_\gK=2x_\gK$ is a {\it characteristic\/} of 
$e_\gK$~\cite[Chap.\,6, \S\,2.1]{t41}. Hence the weighted Dynkin diagram $\gD(\co_\gK)$ is determined 
by the dominant representative in the Weyl group orbit $W{\cdot}h_\gK\subset \te$. Since the 
antidominant representative in $W{\cdot}x_\gK$ is $w_\gK(x_\gK)$ (Prop.~\ref{prop:anti-dom}) and
$\omega_0(x_\gK)=-x_\gK$, the dominant representative is $-w_\gK(x_\gK)$. Therefore, if
$\g\ne\mathfrak{sl}_{2n+1}$, then $W{\cdot}h_\gK\cap\gC=\{2\varpi_j^\vee\}$, where $j$ is 
determined by the condition that $w_\gK(\theta)=-\ap_j$ (cf. Prop.~\ref{prop:fundam}). Thus, if
$\g\ne\mathfrak{sl}_{2n+1}$, then $\co_\gK$ is even and $\gD(\co_\gK)$ has the unique nonzero label 
``2'' that corresponds to $\ap_j$.  

Let $\g=\bigoplus_{i\in\BZ}\g(i)$ be the $\BZ$-grading determined by $h_\gK=2x_\gK$; that is,
$h_\gK$ has the eigenvalue $i$ on $\g(i)$. Then 
\[
  \hot(e_\gK)=\max\{i\mid \g(i)\ne 0\}=2\max\{\gamma(x_\gK)\mid\gamma\in \Delta\} .
\]
Using results of Section~\ref{sect:comb-prop}, we again see that $\hot(e_\gK)\le 4$ and the equality 
occurs if and only if $\theta$ is fundamental.  Note that 
\begin{gather*}
   \dim\g(2)=\#\bigl(\Delta^+_\gK(1)\cup \Delta^-_\gK(1)\bigr)=\#\bigl(\Delta^+_\gK(1)\cup \Delta^+_\gK(-1)\bigr)=
   \#\bigl(\Delta^+_\gK(1)\cup \Delta^+_\gK(2)\bigr) , \\
   \dim\g(4)=\#\Delta^+_\gK(2)=\#\Delta_{\gK}(2)=\dim\Ima(\ad e_\gK)^4 ,
\end{gather*} 
and also $2\dim\g(4)=\dim\ah_\gK$. 

In Tables~\ref{table:O-class} and \ref{table:O-exc}, we point out $\dim\co_{\gK}$ and 
$\gD(\co_\gK)$. For classical cases, we provide the partition $\blb(e_\gK)$; while for the exceptional 
cases, the Dynkin--Bala--Carter notation for orbits is given, see e.g.~\cite[Ch.\,8]{CM}. We also include 
dimensions of the spaces $\g(2)$ and $\g(4)$. 
In Table~\ref{table:O-class}, the unique nonzero numerical mark "$2$" 
corresponds to the simple root $\ap_j$ for all even cases. For $\GR{A}{2j}$, two marks 
``$1$'' correspond to the roots $\ap_j$ and $\ap_{j+1}$. We also assume that $j\ge 2$ in the four orthogonal cases. (For, $\GR{B}{1}=\GR{A}{1}$, $\GR{D}{2}=\GR{A}{1}+\GR{A}{1}$, $\GR{B}{2}=\GR{C}{2}$, and 
$\GR{D}{3}=\GR{A}{3}$.)
\begin{table}[ht]
\caption{The cascade orbits for the classical Lie algebras}   \label{table:O-class}
\begin{center}\begin{tabular}{>{$}l<{$}| >{$}l<{$} >{$}l<{$} c >{$}c<{$} >{$}c<{$}| }
\g & \blb(e_\gK) & \dim\co_{\gK} & $\gD(\co_\gK)$ & \dim\g(2) & \dim\g(4) \\ \hline\hline
\GR{A}{2j-1} & (2^j) & 2j^2 & \rule{0pt}{2.5ex} 
\raisebox{-.3ex}{\begin{tikzpicture}[scale= .65, transform shape]
\node (a) at (1,0) {\bf 0};
\node (b) at (2,0) {$\dots$};
\node (c) at (3,0) {\bf 0};
\node (d) at (4,0) {\bf 2};
\node (e) at (5,0) {\bf 0};
\node (f) at (6,0) {$\dots$};
\node (g) at (7,0) {\bf 0};
\foreach \from/\to in {a/b, b/c, c/d, d/e, e/f, f/g}  \draw[-] (\from) -- (\to);
\end{tikzpicture}} & j^2 & 0
\\
    \GR{A}{2j} & (2^j,1) & 2j^2{+}2j & \rule{0pt}{2.5ex} 
\raisebox{-.3ex}{\begin{tikzpicture}[scale= .65, transform shape]
\node (a) at (1,0) {\bf 0};
\node (b) at (2,0) {$\dots$};
\node (c) at (3,0) {\bf 0};
\node (d) at (4,0) {\bf 1};
\node (e) at (5,0) {\bf 1};
\node (f) at (6,0) {\bf 0};
\node (g) at (7,0) {$\dots$};
\node (h) at (8,0) {\bf 0};
\foreach \from/\to in {a/b, b/c, c/d, d/e, e/f, f/g, g/h}  \draw[-] (\from) -- (\to);
\end{tikzpicture}} & j^2 & 0
\\
    \GR{C}{j} & (2^j) & j^2{+}j & \rule{0pt}{2.5ex} 
\raisebox{-.3ex}{\begin{tikzpicture}[scale= .65, transform shape]
\node (a) at (1,0) {\bf 0};
\node (b) at (2,0) {$\dots$};
\node (c) at (3,0) {\bf 0};
\node (d) at (4.3,0) {\bf 2};
\node (e) at (3.45,0) {\Large $<$};
\draw (3.4, .06) -- +(.6,0);
\draw (3.4, -.06) -- +(.6,0);
\foreach \from/\to in {a/b, b/c}  \draw[-] (\from) -- (\to);
\end{tikzpicture}} & \genfrac{(}{)}{0pt}{}{j+1}{2} & 0
\\ 
    \GR{B}{2j-1} & (3^{j}, 1^{j-1}) & 5j^2{-}3j & \rule{0pt}{2.7ex} 
\raisebox{-.3ex}{\begin{tikzpicture}[scale= .65, transform shape]
\node (a) at (1,0) {\bf 0};
\node (b) at (2,0) {$\dots$};
\node (c) at (3,0) {\bf 0};
\node (d) at (4,0) {\bf 2};
\node (e) at (5,0) {\bf 0};
\node (f) at (6,0) {$\dots$};
\node (g) at (7,0) {\bf 0};
\node (h) at (8.2,0) {\bf 0};
\node (r) at (7.8,0) {\Large $>$};
\foreach \from/\to in {a/b, b/c, c/d, d/e, e/f, f/g}  \draw[-] (\from) -- (\to);
\draw (7.2, .06) -- +(.6,0);
\draw (7.2, -.06) -- +(.6,0);
\end{tikzpicture}} & 2j^2{-}j & \genfrac{(}{)}{0pt}{}{j}{2}  \\
     \GR{D}{2j} & (3^j, 1^{j}) & 5j^2{-}j &  \rule{0pt}{3.5ex} 
\raisebox{-2.5ex}{\begin{tikzpicture}[scale= .65, transform shape]
\node (a) at (1,0) {\bf 0};
\node (b) at (2,0) {$\dots$};
\node (c) at (3,0) {\bf 0};
\node (d) at (4,0) {\bf 2};
\node (e) at (5,0) {\bf 0};
\node (f) at (6,0) {$\dots$};
\node (g) at (7,0) {\bf 0};
\node (h) at (8,.6) {\bf 0};
\node (j) at (8,-.6) {\bf 0};
\foreach \from/\to in {a/b, b/c, c/d, d/e, e/f, f/g, g/h,g/j}  \draw[-] (\from) -- (\to);
\end{tikzpicture}}  & 2j^2& \genfrac{(}{)}{0pt}{}{j}{2} \\
   \GR{B}{2j} & (3^j, 1^{j+1}) & 5j^2{+}j & \rule{0pt}{2.7ex} 
\raisebox{-.3ex}{\begin{tikzpicture}[scale= .65, transform shape]
\node (a) at (1,0) {\bf 0};
\node (b) at (2,0) {$\dots$};
\node (c) at (3,0) {\bf 0};
\node (d) at (4,0) {\bf 2};
\node (e) at (5,0) {\bf 0};
\node (f) at (6,0) {$\dots$};
\node (g) at (7,0) {\bf 0};
\node (h) at (8.2,0) {\bf 0};
\node (r) at (7.8,0) {\Large $>$};
\foreach \from/\to in {a/b, b/c, c/d, d/e, e/f, f/g}  \draw[-] (\from) -- (\to);
\draw (7.2, .06) -- +(.6,0);
\draw (7.2, -.06) -- +(.6,0);
\end{tikzpicture}} & 2j^2{+}j & \genfrac{(}{)}{0pt}{}{j}{2}  \\
      \GR{D}{2j+1} & (3^j, 1^{j+2}) & 5j^2{+}3j & \rule{0pt}{3.5ex} 
\raisebox{-2.5ex}{\begin{tikzpicture}[scale= .65, transform shape]
\node (a) at (1,0) {\bf 0};
\node (b) at (2,0) {$\dots$};
\node (c) at (3,0) {\bf 0};
\node (d) at (4,0) {\bf 2};
\node (e) at (5,0) {\bf 0};
\node (f) at (6,0) {$\dots$};
\node (g) at (7,0) {\bf 0};
\node (h) at (8,.6) {\bf 0};
\node (j) at (8,-.6) {\bf 0};
\foreach \from/\to in {a/b, b/c, c/d, d/e, e/f, f/g, g/h,g/j}  \draw[-] (\from) -- (\to);
\end{tikzpicture}}  & 2j^2{+}2j & \genfrac{(}{)}{0pt}{}{j}{2}  \\ \hline
\end{tabular}
\end{center}
\end{table}

\begin{table}[ht]
\caption{The cascade orbits for the exceptional Lie algebras}   \label{table:O-exc}
\begin{center}\begin{tabular}{>{$}c<{$}|>{$}l<{$} >{$}c<{$} c >{$}c<{$} >{$}c<{$}|}
\g & \co_{\gK} & \dim\co_{\gK} & $\gD(\co_\gK)$ & \dim\g(2) & \dim\g(4) \\ \hline\hline
\GR{E}{6} & \mathsf{A}_2 &  42 & \rule{0pt}{2.5ex} 
\raisebox{-3.2ex}{\begin{tikzpicture}[scale= .65, transform shape]
\node (a) at (0,0) {\bf 0};
\node (b) at (1,0) {\bf 0};
\node (c) at (2,0) {\bf 0};
\node (d) at (3,0) {\bf 0};
\node (e) at (4,0) {\bf 0};
\node (f) at (2,-.9) {\bf 2};
\foreach \from/\to in {a/b, b/c, c/d, d/e, c/f}  \draw[-] (\from) -- (\to);
\end{tikzpicture}} & 20 & 1 \\
\GR{E}{7} & \mathsf{A}_2+3\mathsf{A}_1 & 84 & \rule{0pt}{2.6ex}
\raisebox{-3.2ex}{\begin{tikzpicture}[scale= .65, transform shape]
\node (a) at (0,0) {\bf 0};
\node (b) at (1,0) {\bf 0};
\node (c) at (2,0) {\bf 0};
\node (d) at (3,0) {\bf 0};
\node (e) at (4,0) {\bf 0};
\node (f) at (5,0) {\bf 0};
\node (g) at (3,-.9) {\bf 2};
\foreach \from/\to in {a/b, b/c, c/d, d/e, e/f, d/g}  \draw[-] (\from) -- (\to);
\end{tikzpicture}} & 35 & 7 \\
\GR{E}{8} & 2\mathsf{A}_2 &  156 & \rule{0pt}{2.6ex}
\raisebox{-3.2ex}{\begin{tikzpicture}[scale= .65, transform shape]
\node (h) at (-1,0) {\bf 0};
\node (a) at (0,0) {\bf 0};
\node (b) at (1,0) {\bf 0};
\node (c) at (2,0) {\bf 0};
\node (d) at (3,0) {\bf 0};
\node (e) at (4,0) {\bf 0};
\node (f) at (5,0) {\bf 2};
\node (g) at (3,-.9) {\bf 0};
\foreach \from/\to in {h/a, a/b, b/c, c/d, d/e, e/f, d/g}  \draw[-] (\from) -- (\to);
\end{tikzpicture}} & 64 & 14 \\
\GR{F}{4} & \mathsf{A}_2 & 30 & \rule{0pt}{2.5ex} 
\raisebox{-.5ex}{\begin{tikzpicture}[scale= .65, transform shape]
\node (a) at (0,0) {\bf 0};
\node (b) at (1.1,0) {\bf 0};
\node (c) at (2.5,0) {\bf 0};
\node (d) at (3.6,0) {\bf 2};
\node (e) at (1.55,0) {\Large $<$};
\foreach \from/\to in {a/b,  c/d}  \draw[-] (\from) -- (\to);
\draw (1.5, .06) -- +(.7,0);
\draw (1.5, -.06) -- +(.7,0);
\end{tikzpicture}}  & 14 & 1 \\
\GR{G}{2} & \mathsf{G}_2(a_1) & 10 & \rule{0pt}{2.5ex} 
\raisebox{-.5ex}{\begin{tikzpicture}[scale= .65, transform shape]
\node (b) at (1.1,0) {\bf 0};
\node (c) at (2.5,0) {\bf 2};
\node (e) at (1.55,0) {\Large $<$};
\draw (1.52, .07) -- +(.62,0);
\draw (1.43, 0) -- +(.71,0);
\draw (1.52, -.07) -- +(.62,0);
\end{tikzpicture} } & 4 & 1  \\ \hline
\end{tabular}
\end{center}
\end{table}

Let us summarise main properties of the cascade orbit in all simple $\g$.
\begin{itemize}
\item \ The cascade orbit $\co_\gK$ is even unless $\g$ is of type $\GR{A}{2j}$; this reflects the fact that
$x_\gK\in\mathcal P^\vee$ unless $\g$ is of type $\GR{A}{2j}$.
\item \ If $\theta$ is fundamental, then $\hot(\co_\gK)=4$ and $\co_\gK$ is {\bf not} spherical.
\item \ If $\theta$ is {\bf not} fundamental, then $\hot(\co_\gK)=2$ and $\co_\gK$ is spherical.
Moreover, it appears that $\co_\gK$ is the {\bf maximal} spherical nilpotent orbit in these cases.
\item \ Using the general formulae for the complexity and rank of nilpotent orbits in terms of $\BZ$-gradings~\cite[Sect.\,(2.3)]{p94}, 
one can prove that $c_G(\co_\gK)=2\dim\g(4)=\dim\ah_\gK$ and $r_G(\co_\gK)=\dim\te_\gK=\#\gK$
for all simple $\g$.
\item \ If $\theta$ is fundamental, then ($\co_\gK$ is even and) the node with mark ``$2$'', regarded as a node in the affine Dynkin diagram, determines the {\it Kac diagram\/} of the involution $\sigma_\gK$.
(See \cite[Chap.\,3, \S\,3.7]{t41} for the definition of the Kac diagram of a finite order inner automorphism of $\g$.)
\item \ If $\theta$ is not fundamental and $\co_\gK$ is even, then the node with mark ``$2$'' and the 
extra node in the affine Dynkin diagram, together determine the Kac diagram of the involution 
$\sigma_\gK$.
\end{itemize}

\appendix
\section{The elements of $\gK$ and Hasse diagrams}
\label{sect:tables}

\noindent
Here we provide the lists of cascade elements and the Hasse diagrams of cascade posets $\eus K$ for 
all simple Lie algebras, see Fig.~\ref{fig:An}--\ref{fig:En}. To each node $\beta_j$ in the Hasse diagram, 
the Cartan label of the simple Lie algebra $\g\lg j\rg$ is attached. Recall that $\beta_j$ is the highest root
for $\g\lg j\rg$. If $\g\lg j\rg\simeq  \tri$ and $\beta_j$ is {\sl short}, then we use the Cartan label
$\GRt{A}{1}$. (This happens only for $\GR{B}{2k+1}$ and $\GR{G}{2}$.) It is also assumed that 
$\GR{A}{1}=\GR{C}{1}$.
The main features are:
\begin{itemize}
\item 
$\Pi=\{\ap_1,\dots,\ap_{\rk \g}\}$ and the numbering of $\Pi$ follows~\cite[Table\,1]{VO}, 
\item $\beta_1=\theta$ is always the highest root,
\item The numbering of the $\beta_i$'s  in the lists corresponds to that in the figures.
\end{itemize}
We use the standard $\esi$-notation for the roots of classical Lie algebras, see~\cite[Table\,1]{VO}.

\noindent
{\it\bfseries The list of cascade elements for the classical Lie algebras}:
\begin{description}
\item[$\GR{A}{n}, n\ge 2$]  \ $\beta_i=\esi_i-\esi_{n+2-i}=\ap_i+\dots +\ap_{n+1-i}$ \ ($i=1,2,\dots,\left[\frac{n+1}{2}\right]$);
\item[$\GR{C}{n}, n\ge 1$]  \ $\beta_i=2\esi_i=2(\ap_i+\dots+\ap_{n-1})+\ap_n$ \ ($i=1,2,\dots,n-1$) and 
$\beta_n=2\esi_n=\ap_n$;
\item[$\GR{B}{2n}$, $\GR{D}{2n}$, $\GR{D}{2n+1}$ ($n\ge 2$)]  \ 
$\beta_{2i-1}=\esi_{2i-1}+\esi_{2i}$, $\beta_{2i}=\esi_{2i-1}-\esi_{2i}$ \ ($i=1,2,\dots,n$);
\item[$\GR{B}{2n+1}, n\ge 1$] \ here $\beta_1,\dots,\beta_{2n}$ are as above and $\beta_{2n+1}=\esi_{2n+1}$;
\end{description}

\noindent
For all orthogonal series, we have $\beta_{2i}=\ap_{2i-1}$, $i=1,\dots,n$, while formulae for 
$\beta_{2i-1}$ via $\Pi$ slightly differ for different series. E.g. for $\GR{D}{2n}$ one has
$\beta_{2i-1}=\ap_{2i-1}+2(\ap_{2i}+\dots+\ap_{2n-2})+\ap_{2n-1}+\ap_{2n}$ ($i=1,2,\dots,n-1$)
and $\beta_{2n-1}=\ap_{2n}$.

\noindent
{\it\bfseries The list of cascade elements for the exceptional Lie algebras}:
\begin{description}
\item[$\GR{G}{2}$]  \ $\beta_1=(32)=3\ap_1+2\ap_2, \ \beta_2=(10)=\ap_1$;
\item[$\GR{F}{4}$]   \ $\beta_1=(2432)=2\ap_1+4\ap_2+3\ap_3+2\ap_4,\ \beta_2=(2210),\ \beta_3=(0210),\ \beta_4=(0010)=\ap_3$;
     \item[$\GR{E}{6}$] \  
  $\beta_1=$\raisebox{-2.1ex}{\begin{tikzpicture}[scale= .85, transform shape]
\node (a) at (0,0) {1}; \node (b) at (.2,0) {2};
\node (c) at (.4,0) {3}; \node (d) at (.6,0) {2};
\node (e) at (.8,0) {1}; \node (f) at (.4,-.4) {2};
\end{tikzpicture}}\!\!, 
  $\beta_2=$\raisebox{-2.1ex}{\begin{tikzpicture}[scale= .85, transform shape]
\node (a) at (0,0) {1}; \node (b) at (.2,0) {1};
\node (c) at (.4,0) {1}; \node (d) at (.6,0) {1};
\node (e) at (.8,0) {1}; \node (f) at (.4,-.4) {0};
\end{tikzpicture}}, 
  $\beta_3=$\raisebox{-2.1ex}{\begin{tikzpicture}[scale= .85, transform shape]
\node (a) at (0,0) {0}; \node (b) at (.2,0) {1};
\node (c) at (.4,0) {1}; \node (d) at (.6,0) {1};
\node (e) at (.8,0) {0}; \node (f) at (.4,-.4) {0};
\end{tikzpicture}},
  $\beta_4$\,=\raisebox{-2.1ex}{\begin{tikzpicture}[scale= .85, transform shape]
\node (a) at (0,0) {0}; \node (b) at (.2,0) {0};
\node (c) at (.4,0) {1}; \node (d) at (.6,0) {0};
\node (e) at (.8,0) {0}; \node (f) at (.4,-.4) {0};
\end{tikzpicture}}=\,$\ap_3$;
     \item[$\GR{E}{7}$] \         
  $\beta_1$=\raisebox{-2.1ex}{\begin{tikzpicture}[scale= .85, transform shape]
\node (a) at (0,0) {1}; \node (b) at (.2,0) {2}; \node (c) at (.4,0) {3};
\node (d) at (.6,0) {4}; \node (e) at (.8,0) {3}; \node (f) at (1,0) {2};
\node (g) at (.6,-.4) {2};
\end{tikzpicture}},
  $\beta_2$=\raisebox{-2.1ex}{\begin{tikzpicture}[scale= .85, transform shape]
\node (a) at (0,0) {1}; \node (b) at (.2,0) {2}; \node (c) at (.4,0) {2};
\node (d) at (.6,0) {2}; \node (e) at (.8,0) {1}; \node (f) at (1,0) {0};
\node  at (.6,-.4) {1};
\end{tikzpicture}},
  $\beta_3$=\raisebox{-2.1ex}{\begin{tikzpicture}[scale= .85, transform shape]
\node (a) at (0,0) {1}; \node (b) at (.2,0) {0}; \node (c) at (.4,0) {0};
\node (d) at (.6,0) {0}; \node (e) at (.8,0) {0}; \node (f) at (1,0) {0};
\node (g) at (.6,-.4) {0};
\end{tikzpicture}}=\,$\ap_1$,
  $\beta_4$=\raisebox{-2.1ex}{\begin{tikzpicture}[scale= .85, transform shape]
\node (a) at (0,0) {0}; \node (b) at (.2,0) {0}; \node (c) at (.4,0) {1};
\node (d) at (.6,0) {2}; \node (e) at (.8,0) {1}; \node (f) at (1,0) {0};
\node (g) at (.6,-.4) {1};
\end{tikzpicture}},
  $\beta_5$=\raisebox{-2.1ex}{\begin{tikzpicture}[scale= .85, transform shape]
\node (a) at (0,0) {0}; \node (b) at (.2,0) {0}; \node (c) at (.4,0) {1};
\node (d) at (.6,0) {0}; \node (e) at (.8,0) {0}; \node (f) at (1,0) {0};
\node (g) at (.6,-.4) {0};
\end{tikzpicture}}=\,$\ap_3$,   \\
  $\beta_6$=\raisebox{-2.1ex}{\begin{tikzpicture}[scale= .85, transform shape]
\node (a) at (0,0) {0}; \node (b) at (.2,0) {0}; \node (c) at (.4,0) {0};
\node (d) at (.6,0) {0}; \node (e) at (.8,0) {1}; \node (f) at (1,0) {0};
\node (g) at (.6,-.4) {0};
\end{tikzpicture}}=\,$\ap_5$,
  $\beta_7$=\raisebox{-2.1ex}{\begin{tikzpicture}[scale= .85, transform shape]
\node (a) at (0,0) {0}; \node (b) at (.2,0) {0}; \node (c) at (.4,0) {0};
\node (d) at (.6,0) {0}; \node (e) at (.8,0) {0}; \node (f) at (1,0) {0};
\node (g) at (.6,-.4) {1};
\end{tikzpicture}}=\,$\ap_7$;
      \item[$\GR{E}{8}$] \      
  $\beta_1$=\raisebox{-2.1ex}{\begin{tikzpicture}[scale= .85, transform shape]
\node (a) at (0,0) {2}; \node (b) at (.2,0) {3}; \node (c) at (.4,0) {4};
\node (d) at (.6,0) {5}; \node (e) at (.8,0) {6}; \node (f) at (1,0) {4}; \node (g) at (1.2,0) {2};
\node (h) at (.8,-.4) {3};
\end{tikzpicture}},
  $\beta_2$=\raisebox{-2.1ex}{\begin{tikzpicture}[scale= .85, transform shape]
\node (h) at (-.2,-.0) {0}; \node (a) at (0,0) {1}; \node (b) at (.2,0) {2}; \node (c) at (.4,0) {3};
\node (d) at (.6,0) {4}; \node (e) at (.8,0) {3}; \node (f) at (1,0) {2};
\node (g) at (.6,-.4) {2};
\end{tikzpicture}},
  $\beta_3$=\raisebox{-2.1ex}{\begin{tikzpicture}[scale= .85, transform shape]
\node (h) at (-.2,-.0) {0}; \node (a) at (0,0) {1}; \node (b) at (.2,0) {2}; \node (c) at (.4,0) {2};
\node (d) at (.6,0) {2}; \node (e) at (.8,0) {1}; \node (f) at (1,0) {0};
\node  at (.6,-.4) {1};
\end{tikzpicture}},
  $\beta_4$=\raisebox{-2.1ex}{\begin{tikzpicture}[scale= .85, transform shape]
\node (h) at (-.2,-.0) {0}; \node (a) at (0,0) {1}; \node (b) at (.2,0) {0}; \node (c) at (.4,0) {0};
\node (d) at (.6,0) {0}; \node (e) at (.8,0) {0}; \node (f) at (1,0) {0};
\node (g) at (.6,-.4) {0};
\end{tikzpicture}}=\,$\ap_2$,
  $\beta_5$=\raisebox{-2.1ex}{\begin{tikzpicture}[scale= .85, transform shape]
\node (h) at (-.2,-.0) {0}; \node (a) at (0,0) {0}; \node (b) at (.2,0) {0}; \node (c) at (.4,0) {1};
\node (d) at (.6,0) {2}; \node (e) at (.8,0) {1}; \node (f) at (1,0) {0};
\node (g) at (.6,-.4) {1};
\end{tikzpicture}},  \\
  $\beta_6$=\raisebox{-2.1ex}{\begin{tikzpicture}[scale= .85, transform shape]
\node (h) at (-.2,-.0) {0}; \node (a) at (0,0) {0}; \node (b) at (.2,0) {0}; \node (c) at (.4,0) {1};
\node (d) at (.6,0) {0}; \node (e) at (.8,0) {0}; \node (f) at (1,0) {0};
\node (g) at (.6,-.4) {0};
\end{tikzpicture}}=\,$\ap_4$, 
  $\beta_7$=\raisebox{-2.1ex}{\begin{tikzpicture}[scale= .85, transform shape]
\node (h) at (-.2,-.0) {0}; \node (a) at (0,0) {0}; \node (b) at (.2,0) {0}; \node (c) at (.4,0) {0};
\node (d) at (.6,0) {0}; \node (e) at (.8,0) {1}; \node (f) at (1,0) {0};
\node (g) at (.6,-.4) {0};
\end{tikzpicture}}=\,$\ap_6$,
  $\beta_8$=\raisebox{-2.1ex}{\begin{tikzpicture}[scale= .85, transform shape]
\node (h) at (-.2,-.0) {0}; \node (a) at (0,0) {0}; \node (b) at (.2,0) {0}; \node (c) at (.4,0) {0};
\node (d) at (.6,0) {0}; \node (e) at (.8,0) {0}; \node (f) at (1,0) {0};
\node (g) at (.6,-.4) {1};
\end{tikzpicture}}=\,$\ap_8$;
\end{description}



\begin{figure}[htb]    
\caption{The cascade posets for $\GR{A}{p}$ ($p\ge 2$), $\GR{C}{n}$ ($n\ge 1$), 
$\GR{F}{4}$, $\GR{G}{2}$}  
\label{fig:An}
\vskip1.5ex
\begin{center}
\begin{tikzpicture}[scale= .62] 
\node (2) at (4.5,6.3) {$\beta_1$};
\node (3) at (4.5,4.2) {$\dots$};
\node (4) at (4.5,2.1) {$\beta_{n-1}$};
\node (5) at (4.5,0) {$\beta_{n}$};
\foreach \from/\to in {2/3, 3/4, 4/5} \draw [-,line width=.7pt] (\from) -- (\to);

\draw (6.1,6.3) node { {\color{forest}$\{\GR{A}{2n{-}1}\}$} };
\draw (6.2, 2.1) node { {\color{forest}$\{\GR{A}{3}\}$} };
\draw (6.2,0) node { {\color{forest}$\{\GR{A}{1}\}$} };

\draw (3,3.6) node {$\GR{A}{2n{-}1}$:};
\end{tikzpicture}
\  
\begin{tikzpicture}[scale= .62] 
\node (2) at (4.5,6.3) {$\beta_1$};
\node (3) at (4.5,4.2) {$\dots$};
\node (4) at (4.5,2.1) {$\beta_{n-1}$};
\node (5) at (4.5,0) {$\beta_{n}$};
\foreach \from/\to in {2/3, 3/4, 4/5} \draw [-,line width=.7pt] (\from) -- (\to);

\draw (5.9,6.3) node  { {\color{forest}$\{\GR{A}{2n}\}$} };
\draw (6.1, 2.1) node { {\color{forest}$\{\GR{A}{4}\}$} };
\draw (6.1,0) node     { {\color{forest}$\{\GR{A}{2}\}$} };

\draw (3.2,3.6) node {$\GR{A}{2n}$:};
\end{tikzpicture}
\ 
\begin{tikzpicture}[scale= .62] 
\node (2) at (4.5,6.3) {$\beta_1$};
\node (3) at (4.5,4.2) {$\dots$};
\node (4) at (4.5,2.1) {$\beta_{p{-}1}$};
\node (5) at (4.5,0) {$\beta_p$};
\foreach \from/\to in {2/3, 3/4, 4/5} \draw [-,line width=.7pt] (\from) -- (\to);

\draw (5.9, 6.3) node {{\color{forest}$\{\GR{C}{n}\}$} };
\draw (6, 2.1) node {{\color{forest}$\{\GR{C}{2}\}$} };
\draw (6, 0) node {{\color{forest}$\{\GR{C}{1}\}$} };

\draw (3.3,3.6) node {$\GR{C}{n}$:};
\end{tikzpicture}
\ 
\begin{tikzpicture}[scale= .62] 
\node (2) at (4.5,6.3) {$\beta_1$};
\node (3) at (4.5,4.2) {$\beta_2$};
\node (4) at (4.5,2.1) {$\beta_3$};
\node (5) at (4.5,0) {$\beta_4$};
\foreach \from/\to in {2/3, 3/4, 4/5} \draw [-,line width=.7pt] (\from) -- (\to);

\draw (5.6,6.3) node {{\color{forest}$\{\GR{F}{4}\}$} };
\draw (5.6, 4.2) node {{\color{forest}$\{\GR{C}{3}\}$} };
\draw (5.6, 2.1) node {{\color{forest}$\{\GR{C}{2}\}$} };
\draw (5.6,0) node {{\color{forest}$\{\GR{C}{1}\}$} };

\draw (3.3,3.6) node {$\GR{F}{4}$:};
\end{tikzpicture}
\ 
\begin{tikzpicture}[scale= .62]
\node  (9)  at (9, 4.4) { $\beta_{1}$}; 
\node  (10) at (9, 2.3) { $\beta_{2}$}; 
\node  (11) at (9,0) {}; 
\foreach \from/\to in {9/10} \draw [-,line width=.7pt] (\from) -- (\to);

\draw (10.1,4.4) node {{\color{forest}$\{\GR{G}{2}\}$} };
\draw (10.1,2.3) node {{\color{forest}$\{\GRt{A}{1}\}$} };

\draw (7.8,3.8) node {$\GR{G}{2}$:};
\end{tikzpicture}

\end{center}
\end{figure}


\begin{figure}[ht]    
\caption{The cascade posets for series  $\GR{B}{p}$, $p\ge 3$}  
\label{fig:Bn}
\vskip1.5ex
\begin{center}
\begin{tikzpicture}[scale= .62] 
\node  (1) at (9,10.5) { $\beta_1$};
\node  (2) at (6,8.4) { $\beta_2$};
\node  (3) at (9,8.4) { $\beta_3$};
\node  (4) at (6,6.3) { $\beta_4$};
\node  (5) at (9,6.3) {$\dots$};
\node  (6) at (6,4.2) {$\dots$};
\node  (7) at (9,4.2) { $\beta_{2n-3}$};
\node  (8) at (6,2.1) { $\beta_{2n-2}$};
\node  (9) at (9,2.1) { $\beta_{2n-1}$}; 
\node  (10) at (6,0) { $\beta_{2n}$}; 
\node  (11) at (9,0) { $\beta_{2n+1}$}; 
\foreach \from/\to in {1/2, 1/3, 3/4, 3/5, 5/7,7/8, 7/9, 9/10, 9/11} \draw [-,line width=.7pt] (\from) -- (\to);
\draw[loosely dotted] (5)--(6);

\draw (10.5,10.5) node {{\color{forest}$\{\GR{B}{2n+1}\}$} };
\draw (4.7,8.4) node {{\color{forest}$\{\GR{A}{1}\}$} };
\draw (10.5,8.4) node {{\color{forest}$\{\GR{B}{2n-1}\}$}};
\draw (4.7, 6.3) node {{\color{forest}$\{\GR{A}{1}\}$} };
\draw (10.9, 4.2) node {{\color{forest}$\{\GR{B}{5}\}$}};
\draw (4.4,2.1) node {{\color{forest}$\{\GR{A}{1}\}$} };
\draw (10.9,2.1) node {{\color{forest}$\{\GR{B}{3}\}$}};
\draw (4.4,0) node {{\color{forest}$\{\GR{A}{1}\}$} };
\draw (10.9,0) node {{\color{forest}$\{\GRt{A}{1}\}$} };

\draw (2.9,4.2) node {$\GR{B}{2n+1}$:};
\end{tikzpicture}
\quad
\begin{tikzpicture}[scale= .62] 
\node (1) at (9,10.5) { $\beta_1$};
\node (2) at (6,8.4) { $\beta_2$};
\node (3) at (9,8.4) { $\beta_3$};
\node (4) at (6,6.3) { $\beta_4$};
\node (5) at (9,6.3) {$\dots$};
\node  (6) at (6,4.2)     {$\dots$};
\node (7) at (9,4.2) { $\beta_{2n-3}$};
\node (8) at (6,2.1) { $\beta_{2n-2}$};
\node (9) at (9,2.1) { $\beta_{2n-1}$}; 
\node (10) at (6,0) { $\beta_{2n}$}; 
\foreach \from/\to in {1/2, 1/3, 3/4, 3/5, 5/7,7/8, 7/9, 9/10} \draw [-,line width=.7pt] (\from) -- (\to);
\draw[loosely dotted] (5)--(6);

\draw (10.5,10.5) node {{\color{forest}$\{\GR{B}{2n}\}$} };
\draw (4.7,8.4) node {{\color{forest}$\{\GR{A}{1}\}$} };
\draw (10.5,8.4) node {{\color{forest}$\{\GR{B}{2n-2}\}$} };
\draw (4.7, 6.3) node {{\color{forest}$\{\GR{A}{1}\}$} };
\draw (10.7, 4.2) node {{\color{forest}$\{\GR{B}{4}\}$} };
\draw (4.4,2.1) node {{\color{forest}$\{\GR{A}{1}\}$} };
\draw (10.7,2.1) node {{\color{forest}$\{\GR{B}{2}\}$} };
\draw (4.4,0) node {{\color{forest}$\{\GR{A}{1}\}$} };

\draw (2.6,4.2) node { $\GR{B}{2n}$:};
\end{tikzpicture}
\end{center}
\end{figure}


\begin{figure}[ht]    
\caption{The cascade posets for series $\GR{D}{p}$, $p\ge 4$}   
\label{fig:Dn}
\vskip1.5ex
\begin{center}
\begin{tikzpicture}[scale= .62]
\node  (1) at (9,10.5) { $\beta_1$};
\node  (2) at (5.5,8.4) { $\beta_2$};
\node  (3) at (9,8.4) { $\beta_3$};
\node  (4) at (5.5,6.3) { $\beta_4$};
\node  (5) at (9,6.3) {$\dots$};
\node  (6) at (5.5,4.2) {$\dots$};
\node  (7) at (9,4.2) { $\beta_{2n-5}$};
\node  (8) at (5.5,2.1) { $\beta_{2n-4}$};
\node  (9) at (9,2.1) { $\beta_{2n-3}$}; 
\node  (10) at (5.5,0) { $\beta_{2n-2}$}; 
\node  (11) at (10,0) { $\beta_{2n-1}$}; 
\node  (12) at (12,0) { $\beta_{2n}$}; 
\foreach \from/\to in {1/2, 1/3, 3/4, 3/5, 5/7,7/8, 7/9, 9/10, 9/11,9/12} \draw [-,line width=.7pt] (\from) -- (\to);
\draw[loosely dotted] (5)--(6);

\draw (10.3,10.5) node {{\color{forest}$\{\GR{D}{2n}\}$} };
\draw (4.2,8.4) node {{\color{forest}$\{\GR{A}{1}\}$} };
\draw (10.5,8.4) node {{\color{forest}$\{\GR{D}{2n-2}\}$} };
\draw (4.2, 6.3) node {{\color{forest}$\{\GR{A}{1}\}$} };
\draw (10.9, 4.2) node {{\color{forest}$\{\GR{D}{6}\}$} };
\draw (4,2.1) node {{\color{forest}$\{\GR{A}{1}\}$} };
\draw (10.9,2.1) node {{\color{forest}$\{\GR{D}{4}\}$} };
\draw (4,0) node {{\color{forest}$\{\GR{A}{1}\}$} };
\draw (8.4,0) node {{\color{forest}$\{\GR{A}{1}\}$} };
\draw (13.3,0) node {{\color{forest}$\{\GR{A}{1}\}$} };

\draw (2.8,4.2) node { $\GR{D}{2n}$:};
\end{tikzpicture}
\quad
\begin{tikzpicture}[scale= .62]
\node (1) at (9,10.5) { $\beta_1$};
\node (2) at (6,8.4) { $\beta_2$};
\node (3) at (9,8.4) { $\beta_3$};
\node (4) at (6,6.3) { $\beta_4$};
\node (5) at (9,6.3) {$\dots$};
\node  (6) at (6,4.2)     {$\dots$};
\node (7) at (9,4.2) { $\beta_{2n-3}$};
\node (8) at (6,2.1) { $\beta_{2n-2}$};
\node (9) at (9,2.1) { $\beta_{2n-1}$}; 
\node (10) at (6,0) { $\beta_{2n}$}; 
\foreach \from/\to in {1/2, 1/3, 3/4, 3/5, 5/7,7/8, 7/9, 9/10} \draw [-,line width=.7pt] (\from) -- (\to);
\draw[loosely dotted] (5)--(6);

\draw (10.5,10.5) node {{\color{forest}$\{\GR{D}{2n+1}\}$} };
\draw (4.8,8.4) node {{\color{forest}$\{\GR{A}{1}\}$} };
\draw (10.5,8.4) node {{\color{forest}$\{\GR{D}{2n-1}\}$} };
\draw (4.8, 6.3) node {{\color{forest}$\{\GR{A}{1}\}$} };
\draw (10.7, 4.2) node {{\color{forest}$\{\GR{D}{5}\}$} };
\draw (4.4,2.1) node {{\color{forest}$\{\GR{A}{1}\}$} };
\draw (10.7,2.1) node {{\color{forest}$\{\GR{D}{3}\}$}  };
\draw (4.4,0) node {{\color{forest}$\{\GR{A}{1}\}$} };

\draw (3,4.2) node {$\GR{D}{2n+1}$:};
\end{tikzpicture}
\end{center}
\end{figure}

\vspace{.4cm}

\begin{figure}[ht]    
\caption{The cascade posets for $\GR{E}{6}$, $\GR{E}{7}$, $\GR{E}{8}$}  
\label{fig:En}
\begin{center}
\begin{tikzpicture}[scale= .63] 
\node (2) at (4.5,6.3) {$\beta_1$};
\node (3) at (4.5,4.2) {$\beta_2$};
\node (4) at (4.5,2.1) {$\beta_3$};
\node (5) at (4.5,0) {$\beta_4$};
\foreach \from/\to in {2/3, 3/4, 4/5} \draw [-,line width=.7pt] (\from) -- (\to);

\draw (5.6,6.3) node {{\color{forest}$\{\GR{E}{6}\}$} };
\draw (5.6, 4.2) node {{\color{forest}$\{\GR{A}{5}\}$} };
\draw (5.6, 2.1) node {{\color{forest}$\{\GR{A}{3}\}$} };
\draw (5.6,0) node {{\color{forest}$\{\GR{A}{1}\}$} };

\draw (3,3.6) node {\large $\GR{E}{6}$:};
\end{tikzpicture}
\qquad
\begin{tikzpicture}[scale= .63] 
\node (2) at (10.5,6.3) {$\beta_1$};
\node (3) at (9,4.2) {$\beta_2$};
\node (4) at (10.5,2.1) {$\beta_3$};
\node (5) at (7.5,2.1) {$\beta_4$};
\node (6) at (6,0) {$\beta_5$};
\node (7) at (9,0) {$\beta_6$};
\node (8) at (11.5,0) {$\beta_7$}; 
\foreach \from/\to in {2/3, 3/4, 3/5, 5/6, 5/7, 5/8} \draw [-,line width=.7pt] (\from) -- (\to);

\draw (9.5,6.3) node {{\color{forest}$\{\GR{E}{7}\}$} };
\draw (8, 4.2) node {{\color{forest}$\{\GR{D}{6}\}$} };
\draw (11.5, 2.1) node {{\color{forest}$\{\GR{A}{1}\}$} };
\draw (8.6,2.1) node {{\color{forest}$\{\GR{D}{4}\}$} };
\draw (5,0) node {{\color{forest}$\{\GR{A}{1}\}$} };
\draw (8,0) node {{\color{forest}$\{\GR{A}{1}\}$} };
\draw (10.5,0) node {{\color{forest}$\{\GR{A}{1}\}$} };

\draw (5,3.6) node {\large $\GR{E}{7}$:};
\end{tikzpicture}
\qquad
\begin{tikzpicture}[scale= .63] 
\node   (1) at (9,8.4) {$\beta_1$};
\node   (2) at (10.5,6.3) {$\beta_2$};
\node   (3) at (9,4.2) {$\beta_3$};
\node   (4) at (10.5,2.1) {$\beta_4$};
\node   (5) at (7.5,2.1) {$\beta_5$};
\node   (6) at (6,0) {$\beta_6$};
\node   (7) at (9,0) {$\beta_7$};
\node   (8) at (11.5,0) {$\beta_8$}; 
\foreach \from/\to in {1/2, 2/3, 3/4, 3/5, 5/6, 5/7, 5/8} \draw [-,line width=.7pt] (\from) -- (\to);

\draw (10,8.4) node {{\color{forest}$\{\GR{E}{8}\}$} };
\draw (11.5,6.3) node {{\color{forest}$\{\GR{E}{7}\}$} };
\draw (10, 4.2) node {{\color{forest}$\{\GR{D}{6}\}$} };
\draw (11.5, 2.1) node {{\color{forest}$\{\GR{A}{1}\}$} };
\draw (8.6,2.1) node {{\color{forest}$\{\GR{D}{4}\}$} };
\draw (7.1,0) node {{\color{forest}$\{\GR{A}{1}\}$} };
\draw (10.1,0) node {{\color{forest}$\{\GR{A}{1}\}$} };
\draw (12.5,0) node {{\color{forest}$\{\GR{A}{1}\}$} };

\draw (6,3.6) node {\large $\GR{E}{8}$:};
\end{tikzpicture}
\end{center}
\end{figure}

\end{document}